\documentclass[reqno]{amsart}
\usepackage{amsmath,amsthm,latexsym,amssymb}
\usepackage[mathscr]{eucal}
\usepackage{pgf}
\usepackage{tikz}
\usetikzlibrary{arrows,automata,shapes.geometric,calc}
\tikzset{%
 state/.style={draw, shape=circle, fill=white, inner sep=1pt, minimum size=12pt},
 small state/.style={draw, shape=circle, fill=white, inner sep=1pt, minimum size=6pt},
 big state/.style={draw, shape=circle, fill=white, inner sep=0.5pt, minimum size=15pt},
 gap state/.style={shape=circle, fill=white, inner sep=0.5pt, minimum size=20pt},
 element/.style={draw, shape=circle, fill=white, inner sep=1.5pt},
 standard label/.style={text height=1.5ex,text depth=0.5ex},
 auto}
\theoremstyle{plain}
\newtheorem{theorem}{Theorem}[section]
\newtheorem{corollary}[theorem]{Corollary}
\newtheorem{lemma}[theorem]{Lemma}
\newtheorem{claim}[theorem]{Claim}
\newtheorem{example}[theorem]{Example}

\newtheorem{proposition}[theorem]{Proposition}
\newtheorem*{prop}{Proposition}
\newtheorem*{lem}{Lemma}

\theoremstyle{definition}
\newtheorem{chatexample}[theorem]{Example}
\newtheorem{definition}[theorem]{Definition}
\newtheorem{notation}[theorem]{Notation}
\newtheorem{remark}[theorem]{Remark}

\newcommand{\A}{{\mathbf A}}

\newcommand{\M}{{\mathbf M}}

\newcommand{\dom}[1]{\operatorname{dom}(#1)}

\newcommand{\End}[1]{\operatorname{End}(#1)}

\newcommand{\id}[1]{\operatorname{id}_{#1}}
\newcommand{\ISP}[1]{\operatorname{\mathbb{ISP}}(#1)}
\newcommand{\ks}[1]{\operatorname{kill}(#1)}
\newcommand{\ran}[1]{\operatorname{ran}(#1)}
\newcommand{\sg}[2]{\operatorname{sg}_{#1}(#2)}

\newcommand{\abs}[1]{\lvert#1\rvert}
\newcommand{\comp}{{\setminus}}
\newcommand{\const}[1]{\underline {#1}}
\renewcommand{\emptyset}{\varnothing}
\renewcommand{\ge}{\geqslant}
\renewcommand{\le}{\leqslant}
\newcommand{\ov}[2]{{\vphantom{\vrule height 9.5pt width .5pt depth 2.5pt}%
  \overset{\raise2.5pt\hbox{\smash{$\scriptstyle{#1}$}}}{\lower2.5pt\hbox{\smash{$\scriptstyle{#2}$}}}}}
\renewcommand{\phi}{\varphi}
\newcommand{\rest}[1]{{\upharpoonright}_{#1}}
\newcommand{\set}[2]{\{\,#1 \mid \text{#2}\,\}}
\newcommand{\N}{{\mathbb N}}
\newcommand{\Z}{{\mathbb Z}}
\newcommand{\du}{\overset{.}\cup}

\newcommand{\scrM}{\mathscr M}
\newcommand{\scrbarM}{\overline{\scrM}}
\newcommand{\barY}{\overline{Y}}
\newcommand{\col}[1]{c_{#1}} % was C_{#1}
\newcommand{\newi}{\sigma}
\newcommand{\newip}{\tau}
\newcommand{\newiq}{\rho}
\newcommand{\newir}{\theta}
\newcommand{\newj}{\alpha}
\newcommand{\newjp}{\beta}

\newcommand{\Zp}[1]{\mathbb Z_{p^{#1}}}

\newcommand{\Izero}{B} % was I_0
\newcommand{\Iast}{C} % was J
\newcommand{\Cbar}{{\overline c}} % was D
\newcommand{\tup}[1]{\mathbf{#1}}

\newcommand{\K}{\mathbin{\diamondsuit}}

\begin{document}

\title{Dualizability of automatic algebras}

\author{W. Bentz}
\address[W. Bentz]{Centro de \'Algebra\\Universidade de Lisboa 1649-003 Lisboa\\Portugal}
\email{wolfbentz@googlemail.com}

\author{B. A. Davey}
\address[B. A. Davey]{Department of Mathematics and Statistics\\La Trobe University\\Victoria 3086\\Australia}
\email{B.Davey@latrobe.edu.au}

\author{J. G. Pitkethly}
\address[J. G. Pitkethly]{Department of Mathematics and Statistics\\La Trobe University\\Victoria 3086\\Australia}
\email{J.Pitkethly@latrobe.edu.au}

\author{R. Willard}
\address[R. Willard]{Department of Pure Mathematics\\University of Waterloo\\Waterloo, Ontario N2L 3G1\\Canada}
\email{rdwillar@uwaterloo.ca}

\begin{abstract}
We make a start on one of George McNulty's
\emph{Dozen Easy Problems}: ``Which finite automatic
algebras are dualizable?''
We give some necessary and some sufficient conditions for dualizability.
For example, we prove that a finite automatic algebra is dualizable if its letters
act as an abelian group of permutations on its states.
To illustrate the potential
difficulty of the general problem, we exhibit an infinite
ascending chain $\mathbf A_1 \le \mathbf A_2 \le \mathbf A_3 \le \dotsb$ of finite
automatic algebras that are alternately dualizable and non-dualizable.
\end{abstract}

\maketitle

%%%%%%%%%%%%%%%%%%%%%%%%%%%%%%%%%%%%%%%%%%%%%%%%%%%%%%%%%%%%%%%%%%%%%%%%%%%%%%%%%%%%%%%%%%%%%%%
\section{Introduction}\label{sec:intro}

In this paper, we shall make a start on Problem~6 from George McNulty's \emph{Dozen
Easy Problems}~\cite{dozen}: ``Which finite automatic algebras are
dualizable?''

%%jp
% First sentence below reworded.
An \emph{automatic algebra} is a set with binary operation $\A = \langle Q \cup \Sigma \cup
\{0\}; \cdot\rangle$ that encodes a partial automaton with
state set $Q$ and alphabet~$\Sigma$: the multiplication satisfies
\[
q \cdot a = r \iff q \overset a\to r,
\]
for all $q,r \in Q$ and $a \in \Sigma$; all other products give
the default element~$0  \notin Q \cup \Sigma$. The example
featured in McNulty's problem
is given in Figure~\ref{fig:boozer}.

%%jp
% Extra sentence about Boozer's algebra moved.

%%%%%%%%%%%%%%%%%%%%%%%%%%%%%%%%%%%%
\begin{figure}[ht]
\begin{center}
%%%%%%%%%%%%%%%%%%
\begin{tikzpicture}[baseline=0pt]
  \node[state] (q) at (0,0) {$q$};
  \node[state] (r) at (1,0) {$r$};
  \node[state] (s) at (2,0) {$s$};
  \path (q) edge [->] node {$a$} (r);
  \path (r) edge [loop above] node {$b$} (r);
  \path (r) edge [->] node {$c$} (s);
  \node at (1,-0.75) {$B$};
\end{tikzpicture}
%%%%%%%%%%%%%%%%%%
\qquad
 \begin{tabular}{c|ccccccc}
  $\cdot$ & $0$ & $q$ & $r$ & $s$ & $a$ & $b$ & $c$ \\\hline
      $0$ & $0$ & $0$ & $0$ & $0$ & $0$ & $0$ & $0$ \\
      $q$ & $0$ & $0$ & $0$ & $0$ & $r$ & $0$ & $0$ \\
      $r$ & $0$ & $0$ & $0$ & $0$ & $0$ & $r$ & $s$ \\
      $s$ & $0$ & $0$ & $0$ & $0$ & $0$ & $0$ & $0$ \\
      $a$ & $0$ & $0$ & $0$ & $0$ & $0$ & $0$ & $0$ \\
      $b$ & $0$ & $0$ & $0$ & $0$ & $0$ & $0$ & $0$ \\
      $c$ & $0$ & $0$ & $0$ & $0$ & $0$ & $0$ & $0$
 \end{tabular}
\end{center}
\caption{An example of an automatic algebra}\label{fig:boozer}
\end{figure}
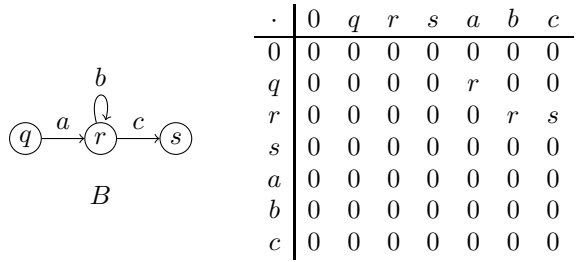
%%%%%%%%%%%%%%%%%%%%%%%%%%%%%%%%%%%

Automatic algebras have been studied mostly as a source of finite algebras with non-finitely based equational theories.
The first finite algebra shown to have a non-finitely based equational theory, due to Lyndon \cite{Lyn} in 1954, is the
automatic algebra based on the automaton $L$ pictured in Figure~\ref{fig:background}.
Automatic algebras were probably first identified as a ``nice class" of algebras by Kearnes and Willard~\cite{KW},
who proved that automatic algebras are 2-step strongly solvable. They also gave a small example of an algebra from
this class whose equational theory is inherently non-finitely based and has residually large models;
it is the automatic algebra based on the automaton
$L_3^\ast$ in Figure~\ref{fig:background}.   Another automatic algebra, based on the automaton $R$ in Figure~\ref{fig:background}, has  the same property and played a supporting role in the
spectacular negative solution of R.~McKenzie to Tarski's finite basis problem \cite{mck-tarski} and the Quackenbush conjecture \cite{mck-quack}.

%%%%%%%%%%%%%%%%%%%%%%%%%%%%%%%%%%%%
\begin{figure}[th]
\begin{center}
%%%%%%%%%%%%%%%%%%
\begin{tikzpicture}
  \node[state] (q) at (0,0) {$q$};
  \node[state] (r) at (1,0) {$r$};
  \node[state] (s) at (2,0) {$s$};
  \path (q) edge [<-] node {$a$} (r);
  \path (q) edge [loop above] node {$a,b,c$} (q);
  \path (r) edge [loop above] node {$b$} (r);
  \path (s) edge [loop above] node {$a,b,c$} (s);
  \path (r) edge [->] node {$c$} (s);
  \node at (1,-0.8) {$L$};
\end{tikzpicture}
%%%%%%%%%%%%%%%%%%
\qquad
\begin{tikzpicture}
  \node[state] (q) at (0,0) {$q$};
  \node[state] (r) at (1,0) {$r$};
  \node[state] (s) at (2,0) {$s$};
  \path (q) edge [loop above] node {$c$} (q);
  \path (r) edge [loop above] node {$a$} (r);
  \path (s) edge [loop above] node {$b$} (s);
  \path (q) edge [bend left=30,<-]  node {$c$} (r);
  \path (r) edge [bend left=30,<-]  node {$a$} (q);
  \path (r) edge [bend left=30,<-]  node {$a$} (s);
  \path (s) edge [bend left=30,<-]  node {$b$} (r);
  \node at (1,-0.8) {$L_3^\ast$};
\end{tikzpicture}
%%%%%%%%%%%%%%%%%%
\qquad
\begin{tikzpicture}
  \node[state] (q) at (0,0) {$q$};
  \node[state] (r) at (1,0) {$r$};
  \path (q) edge [<-] node {$a$} (r);
  \path (r) edge [loop above] node {$b$} (r);
  \path (q) edge [loop above] node {$c$} (q);
  \node at (.5,-0.8) {$R$};
\end{tikzpicture}
\end{center}
\caption{Three more examples}\label{fig:background}
\end{figure}
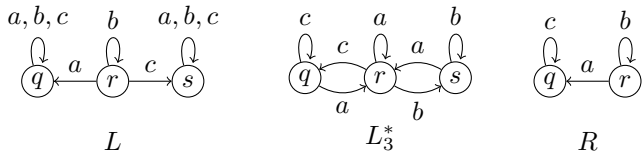
%%%%%%%%%%%%%%%%%%%%%%%%%%%%%%%%%%%

Automatic algebras were first named and explored systematically in the PhD theses of Z.~Sz\'{e}kely \cite{ZS}
and J.~Boozer \cite{B:PhD}, and the article by McNulty, Sz\'ekely and Willard~\cite{MSW}.  These works provide evidence that having a finitely based
equational theory is a relatively rare property amongst finite automatic algebras.
Because of this, the class of finite automatic algebras may also be an interesting source of examples for studying dualizability.

%% Dualizability is another `finiteness condition' on finite algebras.
A finite
algebra is dualizable if it is possible (in a certain natural way) to represent
the algebras in the quasi-variety $\ISP\M$ as algebras of continuous
structure-preserving maps. There is known to be a link between dualizability
and residual smallness~\cite{rlind}: if a finite algebra is dualizable and
generates a congruence-SD($\wedge$) variety, then this variety is residually
small. But it is unclear whether there is any link between dualizability and
finite basedness. The following question, posed over 10 years ago~\cite{DILM},
is still open: `Is every finite dualizable algebra finitely based?'

In this paper, we give general characterizations of dualizability within two
restricted classes of finite automatic algebras: $\abs \Sigma = 1$
(Theorem~\ref{thm:class1}) and $\abs Q = 2$ (Theorem~\ref{thm:class2}). Beyond
these two cases, we give several general necessary conditions for dualizability
(\ref{thm:wc}, \ref{thm:pcomm}, \ref{lem:rankill}) and sufficient conditions
for dualizability (\ref{lem:constants}, \ref{thm:affine}).

%%jp
% Paragraph about how our results relate to the open problem.
% (Moved the Boozer sentence here.)
All the examples of dualizable automatic algebras that we find are known to be
finitely based, by Boozer~\cite[Theorems 1.12 and 1.16]{B:PhD}.
We shall also see that the four non-finitely based automatic algebras that encode $B$, $L$, $L_3^*$ and $R$
are non-dualizable; see Example~\ref{ex:all4}. (The one based on $B$ was shown by Boozer~\cite{B:PhD} to be non-finitely based
but not inherently non-finitely based.)

The most involved proof is that of Theorem~\ref{thm:affine}, which essentially
asserts the following: if $\Sigma$ acts as a coset of a subgroup of an abelian group
of permutations of~$Q$, then the automatic algebra $\M$ is dualizable. We
complement this theorem by giving examples of non-dualizable automatic
algebras where $\Sigma$ acts as a set of commuting
permutations of~$Q$ (\ref{thm:nondcomm}, \ref{ex:nondcomm}).

 To illustrate the potential
difficulty of McNulty's problem, we exhibit an infinite ascending chain $\A_1
\le \A_2 \le \A_3 \le \dotsb$ of finite automatic algebras that are alternately
dualizable and non-dualizable (Example~\ref{ex:chain}). This sort of bad
behavior does not occur in any of the classes of finite algebras where
dualizability has successfully been characterized: for example,
%%jp
% Added extra example at the start (and removed dots at end).
algebras with J\'onsson terms~\cite{DW:first,DHM:NU}, groups~\cite{QS1,QS2,MN}, commutative
rings with unity~\cite{CISSW}, graph algebras~\cite{DILM} and flat graph
algebras~\cite{LMW}. In fact, the only other such chain that has been found so
far is in the class of unary algebras~\cite{idual}.

\begin{notation}
When working with automatic algebras, we usually indicate the groupoid
operation $\cdot$ by concatenation. Note that a groupoid term that is
\emph{not} bracketed from the left like $(\dotsm((x_1x_2)x_3)x_4 \dotsm)x_n$
must be constantly $0$ when interpreted in any automatic algebra and is therefore
equivalent to the term~$xx$. So we \emph{always bracket from the left}. Instead
of writing an expression of the form
\[
( \dotsb (((u \cdot v_1) \cdot v_2) \cdot v_3 ) \dotsb ) \cdot v_n,
\]
we usually just write $uv_1v_2v_3\dotsm v_n$, but we may choose
to write $u \cdot v_1v_2v_3\dotsm v_n$ or $u \cdot v_1 \cdot
v_2 \cdot v_3 \cdot \dotsb \cdot v_n$. We write $u \cdot v^n$
to mean $uvv \dotsm v$, where the $v$ occurs $n$ times. Even if
we use brackets, this does not override the
bracket-from-the-left rule: for example, the expression
$q(ab)^2$ means $qabab$, which really means $(((q\cdot a)\cdot
b)\cdot a) \cdot b$.
\end{notation}

We give a brief definition of `dualizable' in Section~\ref{sec:toolkit}. In the
next section we do not need the definition, just the statement of the Inherent
Non-dualizability Lemma. For a comprehensive introduction to natural duality theory,
see~\cite{NDftWA}.

%%%%%%%%%%%%%%%%%%%%%%%%%%%%%%%%%%%%%%%%%%%%%%%%%%%%%%%%%%%%%%%%%%%%%%%%%%%%%%%%%%%%%%%%%%%%%%%
\section{Two non-dualizability results}\label{sec:easynondual}

In this section, we give two general necessary conditions for an automatic
algebra to be dualizable. We shall use the following standard technique for
proving non-dualizability, due to Davey, Idziak, Lampe and McNulty~\cite{DILM};
see also~\cite[10.5.5]{NDftWA}. Note that a finite algebra $\M$ is
\emph{inherently non-dualizable} if every finite algebra that has $\M$ as a
subalgebra is non-dualizable.

\begin{lemma}[Inherent non-dualizability {\cite{DILM}}]\label{lem:IND}
Let\/ $\M$ be a finite algebra and let\/ $\mu \colon \N \to \N$. Assume there
is a subalgebra $\A$ of\/~$\M^I$, for some set\/~$I$, and an infinite subset\/
$A_0$ of\/ $A$ such that
\begin{enumerate}
 \item for each $n \in \N$ and each congruence $\theta$ on~$\A$ of index at
     most\/~$n$, the equivalence relation $\theta \rest {A_0}$ has a unique
     block of size greater than $\mu(n)$, and
 \item the algebra $\A$ does not contain the element\/ $g$
     of\/ $M^I$ given by $g(i) := a_i(i)$, where $a_i$ is
     any element of the unique infinite block of\/ $\ker
     (\pi_i) \rest {A_0}$.
\end{enumerate}
Then $\M$ is inherently non-dualizable.
\end{lemma}

\begin{notation}\label{not:seq}
When applying the lemma above, we use the following notation to
specify elements of~$M^I$. For all $n \in \N$, all distinct
$i_1,\dotsc,i_n\in I$ and all $u,v_1,\dotsc,v_n\in M$, define
$u \ov{v_1}{i_1} \ov{\dots}{\dots} \,\ov{v_n}{i_n} \in M^I$ by
\begin{equation*}
 u \ov{v_1}{i_1} \ov{\dots}{\dots} \,\ov{v_n}{i_n}(j) =
\begin{cases}
 v_k &\text{if $j=i_k$, for some $k\in \{1,\dotsc,n\}$,}\\
 u &\text{otherwise.}
\end{cases}
\end{equation*}
For $v \in M$, we define $\const v \in M^I$ to be the constant
map with value~$v$.
\end{notation}

\begin{definition}\label{def:whiskcyc}
Fix an automatic algebra~$\M = \langle Q \cup \Sigma \cup
\{0\}; \cdot \rangle$ and let $a \in \Sigma$. We shall say that
the letter $a$ acts as \emph{whiskery cycles} if, for all $q
\in Q$, there exists $n \in \N$ such that $qa = qa^{n+1}$.
Informally, this means that each state in $Q$ is either
\begin{itemize}
 \item in an $a$-cycle,
 \item only one step away from an $a$-cycle, or
 \item not in the domain of~$a$.
\end{itemize}
See Figure~\ref{fig:wc} for an example of a letter acting as whiskery cycles.
\end{definition}

%%%%%%%%%%%%%%%%%%%%%%%%%%%%%%%%%%%%
\begin{figure}[t]
\begin{center}
\begin{tikzpicture}
 \begin{scope}
      \node[small state] (1) at (0,0) {};
      \node[small state] (2) at ($(1)+(0.7,0)$) {};
      \path (1) edge [->, bend left=30] node {} (2);
      \path (2) edge [->, bend left=30] node {} (1);
 \end{scope}
 \begin{scope}[xshift=2.25cm,yshift=0.61cm]
      \node[small state] (1) at (0,0) {};
      \node[small state] (2) at ($(1)+(-60:0.7)$) {};
      \node[small state] (3) at ($(1)+(-120:0.7)$) {};
      \node[small state] (4) at ($(1)+(90:0.7)$) {};
      \node[small state] (5) at ($(1)+(55:0.7)$) {};
      \node[small state] (6) at ($(1)+(125:0.7)$) {};
      \node[small state] (7) at ($(3)+(-135:0.7)$) {};
      \path (1) edge [->, bend left=30] node {} (2);
      \path (2) edge [->, bend left=30] node {} (3);
      \path (3) edge [->, bend left=30] node {} (1);
      \path (4) edge [->] node {} (1);
      \path (5) edge [->] node {} (1);
      \path (6) edge [->] node {} (1);
      \path (7) edge [->] node {} (3);
   \end{scope}
\begin{scope}[xshift=4cm]
      \node[small state] (1) at (0,0) {};
      \node[small state] (4) at ($(1)+(120:0.7)$) {};
      \node[small state] (5) at ($(1)+(60:0.7)$) {};
      \path (1) edge [->, loop below] node {} (1);
      \path (4) edge [->] node {} (1);
      \path (5) edge [->] node {} (1);
   \end{scope}
\begin{scope}[xshift=5.25cm]
      \node[small state] (1) at (0,0) {};
   \end{scope}
\begin{scope}[xshift=6.25cm]
      \node[small state] (1) at (0,0) {};
   \end{scope}
\end{tikzpicture}
\end{center}
\caption{An example of whiskery cycles}\label{fig:wc}
\end{figure}
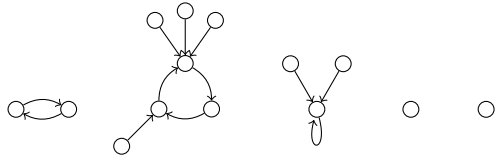
%%%%%%%%%%%%%%%%%%%%%%%%%%%%%%%%%%%%

%%%%%%%%%%%%%%%%%%%%%%%%%%%%%%%%%%%%
\begin{figure}[t]
\begin{center}
\begin{tikzpicture}
  \begin{scope}
    \node at (2,-0.5) {$\mathbf F_0$};
    \node[big state] (q) at (0,0) {$q$};
    \node[big state] (r) at (1,0) {$r$};
    \path (q) edge [->] node {$a$} (r);
  \end{scope}
  \begin{scope}[xshift=4.5cm]
    \node at (5.2,-0.5) {$\mathbf F_m$ ($m \in \N$)};
    \node[big state] (q) at (0,0) {$q$};
    \node[big state] (r) at ($(q)+(1,0)$) {$r$};
    \node[big state] (s1) at ($(r)+(1,0)$) {$s_1$};
    \node[big state] (s2) at ($(s1)+(54:1)$) {$s_2$};
    \node[big state] (s3) at ($(s2)+(-18:1)$) {$s_3$};
    \node[gap state] (s4) at ($(s3)+(-90:1)$) {${\cdot}$};
    \node[big state] (sm) at ($(s1)+(-54:1)$) {$s_m$};
    \node at ($(s4)+(54:0.1)$) {${\cdot}$};
    \node at ($(s4)+(-126:0.1)$) {${\cdot}$};
    \path (q) edge [->] node {$a$} (r);
    \path (r) edge [->] node {$a$} (s1);
    \path (s1) edge [->, bend left=18] node[xshift=1pt, yshift=-3pt] {$a$} (s2);
    \path (s2) edge [->, bend left=18] node[xshift=-3pt] {$a$} (s3);
    \path (s3) edge [-, bend left=15] node {} (s4);
    \path (s4) edge [->, bend left=15] node {} (sm);
    \path (sm) edge [->, bend left=18] node[xshift=1pt, yshift=1pt] {$a$} (s1);
  \end{scope}
\end{tikzpicture}
\end{center}
\caption{Forbidden subalgebras for whiskery cycles}\label{fig:fwc}
\end{figure}
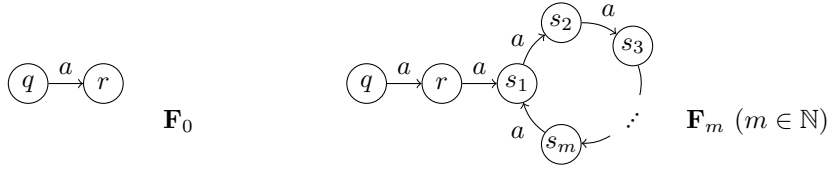
%%%%%%%%%%%%%%%%%%%%%%%%%%%%%%%%%%%%

\begin{lemma}\label{lem:wc}
Let\/ $\M$ be a finite automatic algebra. The following
are equivalent:
\begin{enumerate}
\item each letter acts as whiskery cycles;
\item $\M$ satisfies the quasi-equation\/ $vxx \approx wxx
    \implies vx \approx wx$;
\item for each $m \in \N \cup \{0\}$, the automatic algebra
    $\mathbf F_m$ does not embed into~$\M$; see
    Figure~\ref{fig:fwc}.
\end{enumerate}
\end{lemma}
\begin{proof}[Sketch proof.]
(1)\,$\Rightarrow$\,(2): Assume that each letter acts as
whiskery cycles. Let $v, w, x \in M$ and assume that $vxx =
wxx$ in $\M$. There are $m, n \in \N$ such that $vx = vxx^m$
and $wx = wxx^n$. So $vx = vxx^{mn} = wxx^{mn} = wx$.

(2)\,$\Rightarrow$\,(3): The algebra $\mathbf F_0$ fails the
quasi-equation, as $qaa = 0 = raa$ but $qa = r \ne 0 = ra$. Now
let $m \in \N$. Then $\mathbf F_m$ fails the quasi-equation, as
there exists $k \in \{1,\dotsc,m\}$ such that $qaa = s_1 =
s_kaa$ but $qa = r \ne s_ka$.

(3)\,$\Rightarrow$\,(1): Assume that $a$ does not act as whiskery cycles. Then
there is $q \in Q$ such that $qa \ne qa^{n+1}$, for all $n \in \N$. So $qa \ne
0$. If there is some $k \in \N$ such that $qa^k = 0$, then $\mathbf F_0$ embeds
into~$\M$. Otherwise, since $\M$ is finite, there is some $m \in \N$ such that
$\mathbf F_m$ embeds into~$\M$.
\end{proof}

The next theorem tells us that, if a finite automatic algebra is dualizable,
then every letter must act as whiskery cycles.

\begin{theorem}\label{thm:wc}
Let $\M$ be a finite automatic algebra and let $a \in \Sigma$.
If\/ $a$ does not act as whiskery cycles, then $\M$ is
inherently non-dualizable.
\end{theorem}
\begin{proof}
Fix $m \in \N \cup \{0\}$. By Lemma~\ref{lem:wc}, (3)\,$\Rightarrow$\,(1), it
suffices to prove that the automatic algebra $\mathbf F_m = \langle Q \cup
\Sigma \cup \{0\}; \cdot \rangle$ in Figure~\ref{fig:fwc} is inherently
non-dualizable, where $Q = \{q, r, s_1, s_2, \dotsc, s_m \}$ and $\Sigma =
\{a\}$. (If $m = 0$, then $Q = \{q,r\}$.)

We will use Lemma~\ref{lem:IND} with $\mu \colon \N \to \N$
given by $\mu(n) := n$. Using Notation~\ref{not:seq}, define
$A_0, B \subseteq (F_m)^\N$ by
\[
A_0 := \{\, 0\ov r1\ov ri \mid i \ge 2 \,\} \quad\text{and} \quad B := \{\, 0\ov q1\ov qi\ov qj \mid j > i \ge 2 \,\} \cup \{\, a\ov 0i \mid i \ge 2 \,\},
\]
and define $A := \sg {(\mathbf F_m)^\N}{A_0 \cup B}$.
Condition~\ref{lem:IND}(2) holds, as $g = 0\ov r1$ and
\[
A \subseteq A_0 \cup B \cup \{\, 0\ov r1\ov ri\ov rj \mid j > i \ge 2 \,\} \cup \{0, s_1, s_2, \dotsc, s_m\}^\N.
\]
It remains to establish condition~\ref{lem:IND}(1).

Let $n \in \N$ and let $\theta$ be a congruence on $\A$ of
index at most~$n$. We want to show that $\theta \rest {A_0}$
has a unique block of size greater than~$n$. So consider
disjoint subsets $J$ and $K$ of $\N\comp \{1\}$ with $\abs J =
\abs K = n + 1$. Suppose that each of the two subsets $\{\,
0\ov r1\ov rj \mid j \in J \,\}$ and $\{\, 0\ov r1\ov rk \mid k
\in K \,\}$ of $A_0$ is contained in a block of~$\theta$. It
now suffices to prove that $\{\, 0\ov r1\ov ri \mid i \in J
\cup K \,\}$ is contained in a block of~$\theta$.

The subsets $\{\,a\ov 0j \mid j \in J\,\}$ and $\{\,a\ov 0k \mid k \in K\,\}$
of~$B$ each have size $n + 1$. Since $\theta$ is of index at most~$n$, there
must be distinct $i,j \in J$ and distinct $k,\ell \in K$ such that $a\ov 0i
\equiv_\theta a\ov 0j$ and $a\ov 0k \equiv_\theta a\ov 0\ell$. We now calculate
\[
0\ov r1\ov rj = 0\ov q1\ov qj\ov qk \cdot a\ov 0k
 \equiv_\theta 0\ov q1\ov qj\ov qk \cdot a\ov 0\ell
 = 0\ov r1\ov rj\ov rk
\]
in $\A$. By symmetry, we also have $0\ov r1\ov rk \equiv_\theta
0\ov r1\ov rj\ov rk$. Thus $0\ov r1\ov rj \equiv_\theta 0\ov
r1\ov rk$, and therefore the subset $\{\, 0\ov r1\ov ri \mid i
\in J \cup K \,\}$ is contained in a block of~$\theta$. We have
shown that condition~\ref{lem:IND}(1) holds. Hence $\mathbf
F_m$ is inherently non-dualizable.
\end{proof}

\begin{remark}
The automatic algebra $\mathbf F_0$ is a $3$-nilpotent semigroup, and is
therefore also covered by M.~Jackson's general result~\cite{J:sgps} that
all finite proper $3$-nilpotent semigroups are inherently non-dualizable.
\end{remark}

While having whiskery cycles is necessary for the dualizability of an
automatic algebra, we will see in Example~\ref{ex:rankill} that
it is not sufficient. However, we show in
Section~\ref{sec:class} that a finite automatic algebra with
$\abs \Sigma = 1$ is dualizable if and only if its single
letter acts as whiskery cycles.

The next theorem provides another general necessary condition
for dualizability, which will help with the classification of
$2$-state automatic algebras in Section~\ref{sec:class}.

\begin{theorem}\label{thm:pcomm}
If\/ a finite automatic algebra $\M$ fails the quasi-equation
\[ \tag*{$(\ast)_\phi$}
xy_1y_2\dots y_m
\approx 0 \implies xy_{\varphi(1)}y_{\varphi(2)}\dots
y_{\varphi(m)} \approx 0,
\]
for some $m \in \N$ and some permutation~$\varphi$ of\/
$\{1,2,\dotsc,m\}$, then\/ $\M$ is inherently non-dualizable.
\end{theorem}
\goodbreak
\begin{proof}
For each $m \in \N$, define the condition $C_m$ on $\M$ as
follows:
\begin{itemize}
 \item
the quasi-equation $(\ast)_\phi$ holds in~$\M$, for all permutations~$\varphi$ of $\{1,2,\dotsc,m\}$.
\end{itemize}
Then $C_1$ holds trivially. Now let $m \in \N \cup \{0\}$ and assume that
$C_{m+1}$ holds but $C_{m+2}$ fails. We will prove that $\M$ is inherently
non-dualizable. By Theorem~\ref{thm:wc}, we can assume that every letter of
$\M$ acts as whiskery cycles.

Each permutation of $\{1,2, \dotsc, m+2\}$ can be obtained via
composition from the transposition $(1 \ 2)$ and the cycle $(1
\ 2 \ \dots \ m+2)$. Since $C_{m+2}$ fails, it must
fail with $\varphi = (1 \ 2)$ or $\varphi = (1 \ 2 \ \dots \
m+2)$. We consider these two cases separately.

\smallskip
\emph{Case 1: }$\varphi = (1 \ 2)$. There exist $q \in Q$ and
$a,b,c_1,\dotsc,c_m \in \Sigma$ such that
\[
qabc_1\dotsm c_m = 0 \quad\text{and}\quad r := qbac_1\dotsm c_m \in Q.\label{eq:1}
\]
We start by finding $p \in \N$ and a state $s \in Q$ such that
\begin{enumerate}
 \item $qbb^pac_1\dotsm c_m = r$,
 \item $saa^pac_1\dotsm c_m = r$, and
 \item $qab^pac_1\dotsm c_m = 0$.
\end{enumerate}
We are assuming that each letter of $\M$ acts as whiskery
cycles. So we can fix  $p \in \N$ such that $qb = qbb^p$, and therefore (1)
holds. We must have $qba \in Q$, by the definition of~$r$.
Since $a$ acts as whiskery cycles, it follows that $qba =
saa^pa$, for some $s \in Q$. So (2) holds.

Now suppose, by way of contradiction, that (3) fails. Then
$qab^p \in Q$ and so we can define the states $s_0,s_1,\dotsc,
s_p \in Q$ by
\[
q \overset{a}\to s_0 \overset{b}\to s_1 \overset{b}\to s_2 \overset{b}\to \dotsb \overset{b}\to s_p.
\]
We have $s_pac_1\dotsm c_m = qab^pac_1\dotsm c_m \ne 0$. Since condition
$C_{m+1}$ holds, this implies that $s_pc_1\dotsm c_ma \ne 0$ and so
$s_pc_1\dotsm c_m \ne 0$. Therefore $s_{p-1}bc_1\dotsm c_m \ne 0$, and using
$C_{m+1}$ again it follows that $s_{p-1}c_1\dotsm c_m \ne 0$. Continuing to
argue in this way, we will get $s_1c_1\dotsm c_m \ne 0$. But $s_1 = qab$, by
definition, and so this contradicts our original assumption that $qabc_1\dots
c_m = 0$. Thus (3) holds.

We will prove that $\M$ is inherently non-dualizable using
Lemma~\ref{lem:IND} with the map $\mu \colon \N \to \N$ given
by $\mu(n) := n^2$. Define the sets
\[
A_0 := \{\, r\ov 0i \mid i \in \N \,\} \subseteq M^\N
 \quad\text{and}\quad
A := \{\, v \in M^\N \mid (\exists i) \ v(i) = 0\,\} \cup \Sigma^\N.
\]
Clearly $A$ is a subuniverse of $\M^\N$.
Condition~\ref{lem:IND}(2) holds, as $g = \const r \notin A$.

To check condition~\ref{lem:IND}(1), let $n \in \N$ and let $\theta$ be a
congruence on $\A$ of index at most~$n$. Let $J$ and $K$ be disjoint subsets of
$\N$ with $\abs J = \abs K = n^2 + 1$, and assume that each of the subsets
$\{\, r\ov 0j \mid j \in J \,\}$ and $\{\, r\ov 0k \mid k \in K \,\}$ of $A_0$
is contained in a block of~$\theta$. We want to prove that $\{\, r\ov 0i \mid i
\in J \cup K \,\}$ is contained in a block of~$\theta$.

We consider four subsets of $A$, each of size $n^2 + 1$:
\[
\{\,b\ov 0j \mid j \in J\,\},\quad \{\,b\ov aj \mid j \in J\,\},\quad
\{\,b\ov 0k \mid k \in K\,\},\quad \{\,b\ov ak \mid k \in K\,\}.
\]
(Note that the way $a$ and $b$ were originally chosen ensures they are
distinct.) As $\theta$ is of index at most~$n$, there are distinct $i,j \in J$
and distinct $k,\ell \in K$ such that the following relations hold:
\[
b\ov 0i\equiv_\theta b\ov 0j,\quad b\ov ai\equiv_\theta b\ov aj,\quad
b\ov 0k\equiv_\theta b\ov 0\ell,\quad b\ov ak\equiv_\theta b\ov a\ell.
\]

Define $t := sba^pac_1\dots c_m \in Q\cup \{0\}$. Using equations~(1)--(3), we
calculate
\begin{align*}
r \ov 0i \
 &=\phantom{_\theta} q\ov 0i\ov sk \cdot b\ov ak \cdot (b\ov ak)^p \cdot
  \const a \cdot \const c_1 \cdot \dots \cdot \const c_m\\
 &\equiv_\theta q\ov 0i\ov sk \cdot b\ov a\ell \cdot (b\ov ak)^p \cdot
  \const a \cdot \const c_1 \cdot \dots \cdot \const c_m
 = r \ov 0i\ov tk\ov 0\ell\\
 &=\phantom{_\theta} q\ov 0i\ov sk\ov 0\ell \cdot b\ov 0\ell \cdot (b\ov ak)^p \cdot \const a \cdot \const c_1 \cdot \dots \cdot \const c_m\\
 &\equiv_\theta q\ov 0i\ov sk\ov 0\ell \cdot b\ov 0k \cdot (b\ov ak)^p \cdot \const a \cdot \const c_1 \cdot \dots \cdot \const c_m
 = r \ov 0i\ov 0k\ov 0\ell\\
 &=\phantom{_\theta} q\ov 0i\ov 0k\ov 0\ell \cdot b\ov 0i \cdot (\const b)^p \cdot \const a \cdot \const c_1 \cdot \dots \cdot \const c_m\\
 &\equiv_\theta q\ov 0i\ov 0k\ov 0\ell \cdot b\ov 0j \cdot (\const b)^p \cdot \const a \cdot \const c_1 \cdot \dots \cdot \const c_m
 = r \ov 0i\ov 0j\ov 0k\ov 0\ell
\end{align*}
in $\A$. Using symmetry, we obtain $r \ov 0i \equiv_\theta r
\ov 0i\ov 0j\ov 0k\ov 0\ell \equiv_\theta r \ov 0k$. So
condition~\ref{lem:IND}(1) holds, whence $\M$ is inherently
non-dualizable.

\smallskip
\emph{Case 2: }$\varphi = (1 \ 2 \ \dots \ m+2)$. There are $q \in Q$ and
$a,c_0, c_1, \dotsc, c_m \in \Sigma$ such that $qac_0c_1 \dotsm c_m = 0$ and
$qc_0c_1 \dotsm c_ma \ne 0$. So $qc_0ac_1 \dotsm c_m \ne 0$, as $C_{m+1}$
holds. Thus $\M$ also fails $C_{m+2}$ via the transposition $(1 \ 2)$. So Case
1 applies, whence $\M$ is inherently non-dualizable.
\end{proof}

We can convert the syntactic condition of the previous result
into more concrete conditions. For an automatic algebra $\M$
and for $a \in \Sigma$, define the \emph{domain} of $a$ by
$\dom a := \{\, q \in Q \mid qa \ne 0\,\}$, define the
\emph{range} of $a$ by $\ran a := \{\,qa \mid q \in \dom a\,\}$
and define the set of \emph{kill states} for $a$ by $\ks a := Q
\comp \dom a$.

In the following result, we use the standard notation $\Sigma^*$ for the set
of all words $a_1a_2\dots a_n$ in the alphabet $\Sigma$, where $n \in \N \cup \{0\}$.

\begin{corollary}\label{lem:rankill}
A finite automatic algebra\/ $\M$ is inherently non-dualizable
if there exists $a \in \Sigma$ such that one of the following
conditions holds:
\begin{enumerate}
 \item there is a path from the kill states of\/ $a$ to the
     domain of\/~$a$, that is, there are\/ $q \in \ks a$
     and\/ $w \in \Sigma^*$ such that\/ $qw \in \dom a$;
 \item there is a path from the range of\/ $a$ to the kill
     states of\/~$a$, that is, there are\/ $q \in \ran a$
     and\/ $w \in \Sigma^*$ such that\/ $qw \in \ks a$.
\end{enumerate}
\end{corollary}
\begin{proof}
(1): Assume that $q \in \ks a$ and $w \in \Sigma^*$ with $qw
\in \dom a$. Then $qaw = 0$ but $qwa \ne 0$. So $\M$ is
inherently non-dualizable by Theorem~\ref{thm:pcomm}.

(2): Assume $q \in \ran a$ and $w \in \Sigma^*$ with $qw \in \ks a$. Then $qw
\ne 0$ and $qwa = 0$. By Theorem~\ref{thm:wc}, we can assume $a$ acts as
whiskery cycles. As $q \in \ran a$, this implies that $q = qa^n$, for some $n
\in \N$. So $qa^nw = qw \ne 0$. But $qwa = 0$ and therefore $qwa^n = 0$. Thus
$\M$ is inherently non-dualizable, by Theorem~\ref{thm:pcomm}.
\end{proof}

\begin{chatexample}\label{ex:rankill}
Using the previous corollary, it is easy to check that the three automatic
algebras in Figure~\ref{fig:rankill} are inherently non-dualizable: both
$\mathbf N_1$ and $\mathbf N_2$ have $q \in \ks b$ but $qa \in \dom b$, and so
fail condition~\ref{lem:rankill}(1); the algebra $\mathbf N_3$ has $q \in \ran
b$ but $qa \in \ks b$, and so fails condition~\ref{lem:rankill}(2). We use
these examples in our classification of $2$-state automatic algebras in
Section~\ref{sec:class}.
\end{chatexample}

%%%%%%%%%%%%%%%%%%%%%%%%%%%%%%%%%%%
%%%%%%%%%%%%%%%%%%%%%%%%%%%%%%%%%%%
\begin{figure}
\begin{center}
\begin{tikzpicture}
%% N_1
   \begin{scope}
      \node at (-1,0) {$\mathbf N_1$};
      \node[state] (1) at (0,0) {$q$};
      \node[state] (2) at (1,0) {$r$};
      \path (1) edge [->, bend left] node {$a$} (2);
      \path (2) edge [loop above, very near end, right] node {$a$} (2);
      \path (2) edge [loop below, very near start, right] node {$b$} (2);
   \end{scope}
%% N_2
   \begin{scope}[xshift=4.5cm]
      \node at (-1,0) {$\mathbf N_2$};
      \node[state] (1) at (0,0) {$q$};
      \node[state] (2) at (1,0) {$r$};
      \path (1) edge [<->, bend left] node {$a$} (2);
      \path (2) edge [loop below, very near start, right] node {$b$} (2);
   \end{scope}
%% N_3
   \begin{scope}[xshift=9cm]
      \node at (-1,0) {$\mathbf N_3$};
      \node[state] (1) at (0,0) {$q$};
      \node[state] (2) at (1,0) {$r$};
      \path (1) edge [->, bend left] node {$a$} (2);
      \path (2) edge [loop above, very near end, right] node {$a$} (2);
      \path (1) edge [loop below, very near end, left] node {$b$} (1);
   \end{scope}
\end{tikzpicture}
\end{center}
\caption{Some non-dualizable $2$-state automatic algebras}\label{fig:rankill}
\end{figure}
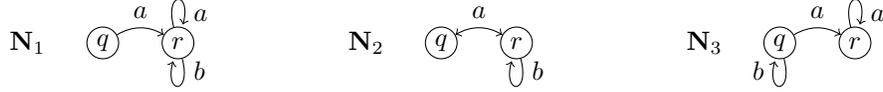
%%%%%%%%%%%%%%%%%%%%%%%%%%%%%%%%%%%
%%%%%%%%%%%%%%%%%%%%%%%%%%%%%%%%%%%

\begin{chatexample}\label{ex:all4}
We have now covered three of the four automatic algebras from the introduction:
the ones based on $B$ and $R$ are non-dualizable by Theorem~\ref{thm:wc};
the one based on $L_3^*$ is non-dualizable by Corollary~\ref{lem:rankill},
as there is a path from $s \in \ran b$ to $q \in \ks b$.
For completeness, we shall check that the automatic algebra based on $L$ is also non-dualizable.

%\smallskip
%
%%%%%%%%%%%%%%%%%%%
%\begin{center}
%\begin{tikzpicture}
%  \node[state] (q) at (0,0) {$q$};
%  \node[state] (r) at (1,0) {$r$};
%  \node[state] (s) at (2,0) {$s$};
%  \path (q) edge [<-] node {$a$} (r);
%  \path (q) edge [loop above] node {$a,c$} (q);
%  \path (s) edge [loop above] node {$a,c$} (s);
%  \path (r) edge [->] node {$c$} (s);
%\end{tikzpicture}
%\end{center}
%%%%%%%%%%%%%%%%%%%
%
%\smallskip

Consider the subalgebra $\M = \langle \{ q, r, s \} \cup \{a, c \} \cup \{0\}; \cdot \rangle$ of Lyndon's automatic algebra.
We will use Lemma~\ref{lem:IND} with $\mu(n) := n$.
Define $A_0 \subseteq A \subseteq M^\N$ by
\[
A_0 := \{\, q\ov si \mid i \in \N \,\} \quad\text{and}\quad A:= (Q^{\N}\comp \{q,r\}^{\N}) \cup \Sigma^{\N} \cup \{\underline{0}\}.
\]
Then $A$ forms a subalgebra $\A$ of $\M^\N$, and condition~\ref{lem:IND}(2) holds as $g = \underline{q} \not\in A$.

To see that condition~\ref{lem:IND}(1) holds, let $n \in \N$ and let $\theta$
be a congruence on $\A$ of index at most~$n$. Let $J$ and $K$ be disjoint subsets
of $\N$ of size $n + 1$, and assume that the sets $\{\, q\ov
sj \mid j \in J \,\}$ and $\{\, q\ov sk \mid k \in K \,\}$ are
each contained in a block of~$\theta$. As $\theta$ is of index at most $n$, there are distinct $i,j \in J$ and distinct
$k,\ell \in K$ such that $c\ov ai \equiv_{\theta} c \ov aj$ and $c\ov ak
\equiv_{\theta} c \ov a\ell$. Therefore
\[
q \ov si = q \ov si \ov rk \cdot c \ov ak \equiv_{\theta} q \ov si \ov rk \cdot c \ov a\ell = q \ov si \ov sk.
\]
By symmetry, we get $q \ov si  \equiv_{\theta} q \ov si \ov sk \equiv_{\theta}
q \ov sk $. So $\M$ is inherently non-dualizable.
\end{chatexample}

%%%%%%%%%%%%%%%%%%%%%%%%%%%%%%%%%%%%%%%%%%%%%%%%%%%%%%%%%%%%%%%%%%%%%%%%%%%%%%%%%%%%%%%%%%%%%%%
\section{Dualizability toolkit}\label{sec:toolkit}

In this section, we give some general definitions and results that will be
helpful in our dualizability proofs in the following two sections. We do not
need to define \emph{dualizable} in full generality. Instead we define a
simpler sufficient condition.

%We start with a very quick definition of \emph{dualizable}.

\begin{definition}
Fix a finite algebra~$\M$. Consider a function $f \colon \hom(\A,\M) \to M$,
where $\A$ is any algebra of the same type as~$\M$.
\begin{itemize}
 \item The function $f$ is called an \emph{evaluation} if there exists $a
     \in A$ with $f(x) = x(a)$, for all $x \colon \A \to \M$.
 \item For $k \in \N$, the function $f$ is \emph{$k$-locally an evaluation}
     if its restriction $f \rest X$ agrees with an evaluation, for all $X
     \subseteq \hom(\A,\M)$ with $\abs X \le k$.
\end{itemize}
Now, for $k \in \N$, we say that $\M$ is
\emph{$k$-dualizable} provided the following holds:
\begin{itemize}
 \item for each finite algebra $\A \in \ISP\M$ and each
     function $f \colon \hom(\A,\M) \to M$, if $f$ is
     $k$-locally an evaluation, then $f$ is an evaluation.
\end{itemize}
\end{definition}

In fact, this definition uses the Duality Compactness
Theorem~\cite{dct,rank,CISSW}; see also~\cite[2.2.11]{NDftWA}. In this paper,
we always establish that a finite automatic algebra is dualizable by
showing that it is $k$-dualizable, for some $k \in \N$. But there are
dualizable algebras that are not $k$-dualizable, for any $k \in \N$~\cite{unc}.

\begin{definition}
Let $\M$ be an algebra and let $k \in \N$. A $k$-ary relation $r$ on $M$ is
\emph{compatible} with $\M$ if it is a subuniverse of $\M^k$. A partial
operation on $M$ is \emph{compatible} with $\M$ if its graph is a compatible
relation on~$\M$ (or, equivalently, its domain is a compatible relation and it
is a homomorphism).
\end{definition}

Note that relations on $M$ can be interpreted pointwise on the subset
$\hom(\A,\M)$ of~$M^A$, and $\hom(\A,\M)$ is closed under every compatible
partial operation on~$\M$. We require the following easy but useful lemma
(see~\cite[10.5.1]{NDftWA} or \cite[1.4.4]{DUAaB}).

\begin{lemma}[Preservation]\label{lem:pres}
Let $k \in \N$ and let $f \colon \hom(\A,\M) \to M$, where $\M$ is a finite algebra
and\/ $\A \in \ISP\M$. Then $f$ is $k$-locally an evaluation if and only if\/ $f$
preserves every $k$-ary compatible relation on\/~$\M$.
\end{lemma}

We also use the fact that two different automatic algebras that generate the
same quasi-variety are either both dualizable or both not.

\begin{theorem}[Independence of the generator~\cite{DWind,Sind}]\label{thm:indep}
Let $\M$ and $\mathbf N$ be finite algebras and assume that
$\ISP\M = \ISP{\mathbf N}$. If\/ $\M$ is dualizable, then so
is~$\mathbf N$.
\end{theorem}

\begin{remark}\label{rem:easy}
We can quickly eliminate some `trivial' cases from our study of automatic
algebras. If $Q = \emptyset$ or $\Sigma = \emptyset$, then the automatic
algebra $\M$ is a zero-semigroup and therefore dualizable (see~\cite[Exercise 3.7]{NDftWA}).
Also, since different automatic algebras that generate the same quasi-variety
are equivalent as far as dualizability is concerned, we can make the following
restrictions on the automatic algebras we consider.
\begin{enumerate}
 \item \emph{No `totally undefined' letters.} Assume $Q \ne \emptyset$ and
     there is $a \in \Sigma$ with $\dom a = \emptyset$. Then $\M$ generates
     the same quasi-variety as its subalgebra $\mathbf N$ with universe $N
     := M \comp \{a\}$. (To see this, choose $q \in Q$ and define the
     embedding $\phi \colon \M \to \mathbf N^2$ by $x \mapsto (x,0)$, for
     all $x \in N$, and $a \mapsto (0,q)$.)
 \item \emph{No `repeated' letters.} Assume there are distinct $a,b \in
     \Sigma$ such that $qa = qb$, for all $q \in Q$. Then $\M$ generates
     the same quasi-variety as its subalgebra on $N := M \comp \{a\}$.
     (Define $\phi \colon M \to N^2$ by $x \mapsto (x,0)$, for all $x \in
     N$, and $a \mapsto (b,b)$.)
 \item \emph{No `isolated' states.} Assume $\Sigma \ne \emptyset$ and $q
     \in Q$ with $q \notin \dom a \cup \ran a$, for all $a \in \Sigma$.
     Then $\M$ generates the same quasi-variety as its subalgebra on $N :=
     M \comp \{q\}$. (Choose $a \in \Sigma$ and define $\phi \colon M \to
     N^2$ by $x \mapsto (x,0)$, for all $x \in N$, and $q \mapsto (0,a)$.)
  \item \emph{No `redundant' states.} Assume there are distinct $q,r \in
      Q$ with $q \notin \ran a$ and $qa = ra$, for all $a \in \Sigma$.
      Then $\M$ generates the same quasi-variety as its subalgebra on $N :=
      M \comp \{q\}$. (Define $\phi \colon M \to N^2$ by $x \mapsto (x,0)$,
      for all $x \in N$, and $q \mapsto (r,r)$.)
\end{enumerate}
\end{remark}

Assume $\M$ is a finite automatic algebra with $M = Q \cup \Sigma \cup \{0\}$.
We say that a subset $C$ of $Q$ is a \emph{component} of $\M$ if it is a
connected component of the underlying graph of the partial automaton (that is, the
graph $\langle Q; \sim\rangle$ with $q \sim r$ if and only if $qa = r$ or $ra =
q$, for some $a \in \Sigma$). In this case, we call the subalgebra of $\M$ with
universe $C \cup \Sigma \cup \{0\}$ a \emph{component subalgebra} of $\M$. If
$\M$ has only one component, then we say that it is \emph{connected}.

The following easy fact will be useful in combination with
independence of the generator (Theorem~\ref{thm:indep}).

\begin{lemma} \label{lm:quasiv}
Let $\M$ and $\mathbf N$ be finite automatic algebras. Assume every
component subalgebra of $\M$ belongs to $\ISP{\mathbf N}$, and vice versa. Then
$\ISP\M = \ISP{\mathbf N}$.
\end{lemma}
\begin{proof}
Let $\M_1,\dots,\M_n$ be the component subalgebras of~$\M$. Using symmetry, it suffices to show that
$\M \in \ISP{\{\M_1,\dots,\M_n\}}$. For each $i \in \{1,\dots,n\}$, let $C_i$ denote the component of $\M$
corresponding to~$\M_i$. So $Q = C_1 \cup \dots \cup C_n$. Now define the map $\phi \colon M \to M_1 \times \dots
\times M_n$ by
\[
\phi(v) =
 \begin{cases}
 (0,\dots,0,\overset iv,0,\dots,0) &\text{if $v \in C_i$, for some $i \in \{1,\dots,n\}$,}\\
 (v,v,\dots,v) &\text{if $v \in \Sigma \cup\{0\}$.}
 \end{cases}
\]
Then $\phi$ is an embedding from $\M$ into $\M_1 \times \dots \times \M_n$.
\end{proof}

We now define some compatible operations and relations on automatic algebras
that will be used in the following two sections.

\begin{definition}\label{def:gst}
Let $\M$ be any automatic algebra. For all $u,v \in M$ such that $\{u,v\} \cap
\Sigma \ne \emptyset$, we can define the homomorphism $g_{u,v} \colon \M^2 \to
\M$ by
\[
g_{u,v}(x,y) :=
 \begin{cases}
 u &\text{if\/ $(x,y) = (u,v)$,}\\
 0 &\text{otherwise.}
 \end{cases}
\]
To check $g_{u,v}$ is a homomorphism, let $w,x,y,z \in M$. Then $g_{u,v}(w\cdot
x,y\cdot z) = 0$, as $w\cdot x,y\cdot z \in Q \cup \{0\}$, and $g_{u,v}(w,y)
\cdot g_{u,v}(x,z) = 0$, as $\{0,u\} \cdot \{0,u\} = \{0\}$.
\end{definition}

The following general lemma is an application of the `binary homomorphism' techniques
introduced in~\cite{binhoms}; see also~\cite[Section 2.2]{DUAaB}.

\begin{lemma}\label{lem:gst}
Let $\M$ be a finite algebra and let $f \colon \hom(\A,\M) \to M$,
for some finite $\A \in \ISP\M$.
Assume there exists\/ $u \in \ran f$ such that, for all\/ $v \in M$,
there is a homomorphism $g_{u,v} \colon \M^2 \to \M$ satisfying
\[
(\forall x,y \in M)\quad g_{u,v}(x,y) = u \iff (x,y) = (u,v).
\]
If\/ $f$ is $3$-locally an evaluation, then $f$ is an
evaluation.
\end{lemma}
\begin{proof}
Assume $f$ is $3$-locally an evaluation.
By Lemma~\ref{lem:pres}, the map $f$ preserves all ternary compatible relations
on $\M$ and therefore preserves~$g_{u,u}$. We have $g_{u,u}^{-1}(u) =
\{(u,u)\}$ and so, by the Strong Idempotents Lemma~\cite[Lemma 12]{binhoms},
the map $f$ agrees with evaluation at some $a \in A$ on $f^{-1}(u)$. Now let $v
\in M$. Using $g_{u,v}$ and the First GST Lemma~\cite[Lemma 17]{binhoms}, it
follows that $f$ also agrees with evaluation at $a$ on $f^{-1}(v)$. Thus $f$ is
evaluation at~$a$.
\end{proof}

The previous lemma and Definition~\ref{def:gst} yield the following corollary,
which will be used to cover one case in both of our main dualizability proofs.

\begin{corollary}\label{cor:gst}
Let $\M$ be a finite automatic algebra and let $f \colon \hom(\A,\M) \to M$,
for some finite $\A \in \ISP\M$, with $\ran f \cap \Sigma \ne \emptyset$.
If\/ $f$ is $3$-locally an evaluation, then $f$ is an
evaluation.
\end{corollary}

\begin{definition}\label{def:order}
Again, let $\M$ be any automatic algebra. We define an order on
$M$ by ${\sqsubseteq} := \Delta_M \cup (\{0\} \times Q)$; see
the diagram below.

%%%%%%%%%%%%%%%%%%%%%%%%%%%%%%%%%%%%%%%%%%%%%%%%%%%%%%%%%%%%%%%%%%%%%%%%%%%%%%%%%%%%%%%%%%
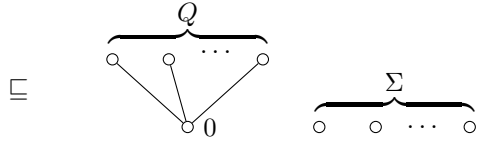
\begin{center}
\begin{tikzpicture}
 \node at (0,0.5) {$\sqsubseteq$};
 %% Q
 \begin{scope}[xshift=2.25cm]
   \node[element] (0) at (0,0) {};
      \node at (0.3,0) {$0$};
   \node[element] (q) at (-1,0.9) {};
   \node[element] (r) at (-0.25,0.9) {};
   \node at (0.4,1) {$\dots$};
   \node[element] (s) at (1,0.9) {};
      \node at (0,1.2) {$\overbrace{\qquad\qquad\qquad}$};
      \node at (0,1.5) {$Q$};
   \path (q) edge (0);
   \path (r) edge (0);
   \path (s) edge (0);
 \end{scope}
 %% Sigma
 \begin{scope}[xshift=5cm]
   \node[element] (a) at (-1,0) {};
   \node[element] (b) at (-0.25,0) {};
   \node[element] (c) at (1,0) {};
   \node at (0.4,0) {$\dots$};
   \node at (0,0.3) {$\overbrace{\qquad\qquad\qquad}$};
   \node at (0,0.6) {$\Sigma$};
 \end{scope}
\end{tikzpicture}
\end{center}
%%%%%%%%%%%%%%%%%%%%%%%%%%%%%%%%%%%%%%%%%%%%%%%%%%%%%%%%%%%%%%%%%%%%%%%%%%%%%%%%%%%%%%%%%%

\noindent The induced partial join operation is a homomorphism $\sqcup \colon
\mathbf D \to \M$, where the domain $\mathbf D$ is the subalgebra of $\M^2$
with universe $D := {\sqsubseteq} \cup {\sqsupseteq}$. To check this claim, it
suffices to show that $r:= \mathrm{graph}(\sqcup)$ is a subuniverse of $\M^3$.
Let $\vec x, \vec y \in r$. Since $\hat 0 \in r$, we can assume that $\vec x
\cdot \vec y \ne \hat 0$. So there must be $q \in Q$ and $a \in \Sigma$ such
that $\vec x \in \{ (0,q,q), (q,0,q), (q,q,q) \}$ and $\vec y = (a,a,a)$.
Therefore $\vec x \cdot \vec y \in r$, as required.
\end{definition}

\begin{definition}\label{def:quasi}
Now let $\M$ be a \emph{total automatic algebra} (that is, an automatic algebra
such that $\dom a = Q$, for every $a \in \Sigma$). Define a quasi-order on $M$
by ${\preccurlyeq} := Q^2 \cup \Sigma^2 \cup (\{0\} \times M)$; see the diagram
below.

%%%%%%%%%%%%%%%%%%%%%%%%%%%%%%%%%%%%%%%%%%%%%%%%%%%%%%%%%%%%%%%%%%%%%%%%%%%%%%%%%%%%%%%%%%
\smallskip
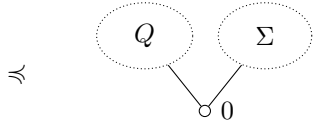
\begin{center}
\begin{tikzpicture}
 \node at (0,0.5) {$\preccurlyeq$};
 \begin{scope}[xshift=2.5cm]
   \node[element] (0) at (0,0) {};
      \node at (0.3,0) {$0$};
   \node[ellipse,draw,densely dotted,minimum size=25pt] (Q) at (-0.8,1) {\phantom{ii}$Q$\phantom{ii}};
   \node[ellipse,draw,densely dotted,minimum size=25pt] (S) at (0.8,1) {\phantom{ii}$\Sigma$\phantom{ii}};
   \path (Q) edge (0);
   \path (S) edge (0);
 \end{scope}
\end{tikzpicture}
\end{center}
%%%%%%%%%%%%%%%%%%%%%%%%%%%%%%%%%%%%%%%%%%%%%%%%%%%%%%%%%%%%%%%%%%%%%%%%%%%%%%%%%%%%%%%%%%

\noindent Then we can define an associated quasi-meet operation by
\[
u \curlywedge v :=
 \begin{cases}
 u &\text{if $(u, v) \in Q^2 \cup \Sigma^2$,}\\
 0 &\text{otherwise.}
 \end{cases}
\]
To see that $\curlywedge \colon \M^2 \to \M$ is a homomorphism, let $x,y,u,v
\in M$. We want to show that $(x\cdot u) \curlywedge (y \cdot v) =
(x\curlywedge y) \cdot (u \curlywedge v)$. We can assume that $x,y \in Q$ and
$u,v \in \Sigma$, since otherwise both sides evaluate to~$0$. As $\M$ is total,
we have $x\cdot u, y \cdot v \in Q$. So both sides evaluate to $x \cdot u$.
\end{definition}

%%%%%%%%%%%%%%%%%%%%%%%%%%%%%%%%%%%%%%%%%%%%%%%%%%%%%%%%%%%%%%%%%%%%%%%%%%%%%%%%%%%%%%%%%%%%%%%
\section{Letters acting as constants}\label{sec:easydual}

In this section, we show that a finite automatic algebra is dualizable if every
letter $a \in \Sigma$ acts as a constant unary operation on~$Q$. This result
will be used in Section~\ref{sec:class}, where we describe which $2$-state
automatic algebras are dualizable.

Note that, if every letter acts as a constant, then the automatic algebra
satisfies the equation $z  \cdot yx \approx z \cdot xyx$, and therefore has a
finitely based equational theory by Boozer~\cite[Theorem 1.16]{B:PhD}.

\begin{theorem}\label{lem:constants}
Let $\M$ be a finite total automatic algebra such that each letter is constant
on~$Q$. Then\/ $\M$ is dualizable.
\end{theorem}
\begin{proof}
We can assume that $Q = \{q_1,\dotsc,q_n\}$ and $\Sigma = \{a_1, \dotsc,
a_n\}$, for some $n \in \N$, where each letter $a_i$ is constant with
value~$q_i$. (Use (2) and (4) from Remark~\ref{rem:easy}.
In fact, we could restrict to the case $n \le 2$.)

Let $\A$ be a finite algebra in $\ISP\M$ and define
\[
D(\A) := \hom(\A,\M) \subseteq M^A.
\]
Assume that $f \colon D(\A) \to M$ is $4$-locally an evaluation. We aim to
prove that $f$ is an evaluation.

If $\ran f = \{0\}$, then $f$ is given by evaluation at~$0^\A$. If $\ran f \cap
\Sigma \ne \emptyset$, then $f$ is an evaluation, by Corollary~\ref{cor:gst}.
So we can assume that $\ran f \subseteq Q \cup \{0\}$
and, without loss of generality, that $q_1 \in \ran f$.

We claim that the meet operation shown below is a homomorphism $\wedge \colon
\M^2 \to \M$.

%%%%%%%%%%%%%%%%%%%%%%%%%%%%%%%%%%%%%%%%%%%%%%%%%%%%%%%%%%%%%%%%%%%%%%%%%%%%%%%%%%%%%%%%%%
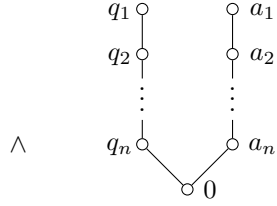
\begin{center}
\begin{tikzpicture}
 \node at (0,0.6) {$\wedge$};
 \begin{scope}[xshift=2.25cm]
   \node[element] (0) at (0,0) {};
      \node at (0.3,0) {$0$};
   \node[element] (1) at (-0.6,2.4) {};
      \node[standard label] at (-0.9,2.4) {$q_1$};
   \node[element] (2) at (-0.6,1.8) {};
      \node[standard label] at (-0.9,1.8) {$q_2$};
   \node[text height=2ex] (dL) at (-0.6,1.2) {$\vdots$};
   \node[element] (n) at (-0.6,0.6) {};
      \node[standard label] at (-0.9,0.6) {$q_n$};
   \node[element] (a1) at (0.6,2.4) {};
      \node[standard label] at (1,2.4) {$a_1$};
   \node[element] (a2) at (0.6,1.8) {};
      \node[standard label] at (1,1.8) {$a_2$};
   \node[text height=2ex] (dR) at (0.6,1.2) {$\vdots$};
   \node[element] (an) at (0.6,0.6) {};
      \node[standard label] at (1,0.6) {$a_n$};
   \path (1) edge (2);
   \path (2) edge (dL);
   \path (dL) edge (n);
   \path (n) edge (0);
   \path (a1) edge (a2);
   \path (a2) edge (dR);
   \path (dR) edge (an);
   \path (an) edge (0);
 \end{scope}
\end{tikzpicture}
\end{center}
%%%%%%%%%%%%%%%%%%%%%%%%%%%%%%%%%%%%%%%%%%%%%%%%%%%%%%%%%%%%%%%%%%%%%%%%%%%%%%%%%%%%%%%%%%

\noindent To check this claim, let $x,y,u,v \in M$. We want to show that $(x
\cdot u) \wedge (y \cdot v) = (x \wedge y) \cdot (u \wedge v)$. We can assume
$x,y \in Q$ and $u,v \in \Sigma$, since otherwise both sides evaluate to~$0$.
It is now easy to check that both sides evaluate to~$q_m$, where $m$ is the
largest index such that $a_m \in \{u,v\}$.

So $D(\A)$ is a semilattice under the pointwise operation~$\wedge$ and the map
$f$ is a semilattice homomorphism (as $f$ is $3$-locally an evaluation). Since
$D(\A)$ is finite and $q_1 \in \ran f$, the set $f^{-1}(q_1)$ is a principal filter of $D(\A)$. Let
$w \colon \A \to \M$ denote the least element of $f^{-1}(q_1)$ and define
\[
A_1 := w^{-1}(q_1) \subseteq A.
\]
Since $f$ is $1$-locally an evaluation, we know that $A_1 \ne \emptyset$. We
will be needing the following fact about~$A_1$.

\medskip

\noindent \emph{Claim. $f(x) = x(\sigma)$, for all\/ $x \in f^{-1}(Q)$ and
all\/ $\newi \in A_1$.}

\smallskip

\noindent Let $x \in f^{-1}(Q) \subseteq D(\A)$ and let $\newi \in A_1 =
w^{-1}(q_1)$. Say that $f(x) = q_i$. There is an automorphism $\varphi$ of $\M$
such that $\varphi(q_i) = q_1$. Since $f$ is $2$-locally an evaluation, it
preserves~$\phi$. So $\varphi \circ x \in D(\A)$ with $f(\varphi \circ x) =
\varphi(f(x)) = q_1$. Thus $\varphi \circ x \ge w$ in the semilattice $D(\A)$.
It follows that $\varphi \circ x(\newi) \ge w(\newi) = q_1$ and therefore
$\varphi \circ x(\newi) = q_1$. Hence $x(\newi) = q_i$, as required.

\medskip

Now suppose, by way of contradiction, that $f$ is not an evaluation. We
consider two cases.

\medskip

\noindent \emph{Case 1: $w^{-1}(a_1) = \emptyset$.} Let $\vee \colon
\{0,q_1\}^2 \to \{0,q_1\}$ denote the join operation coming from the order $0 <
q_1$. Define the ternary partial operation $h$ on $M$ with domain $D := (M\comp
\{a_1\}) \times M^2$ by
\[
h(x,y,z) :=
 \begin{cases}
 y \vee z &\text{if $x = q_1$ and $y,z \in \{0,q_1\}$,}\\
 0 &\text{otherwise.}
 \end{cases}
\]
Then it is easy to check that $\mathbf D \le \M^3$ and that $h \colon \mathbf D
\to \M$ is a homomorphism.

Let $\newi \in A_1$. We are supposing that $f$ is not given by evaluation
at~$\newi$. Since $\ran f \subseteq Q \cup \{0\}$, it follows from the claim
above that there is $x_\newi \in f^{-1}(0)$ with $0 = f(x_\newi) \ne
x_\newi(\newi)$. We can assume that $x_\newi(\newi) \in Q$. (If $x_\newi(\newi)
= b \in \Sigma$, then use Definition~\ref{def:gst} and replace $x_\newi$ by
$g_{q_1,b}(w,x_\newi)$.) Using Definition~\ref{def:quasi}, set $y_\newi := w
\curlywedge x_\newi \in D(\A)$. Then $f(y_\newi) = f(w) \curlywedge f(x_\newi)
= 0$, with $y_\newi(\newi) = q_1$ and $y_\newi(A_1) \subseteq \{0,q_1\}$.

Now enumerate $A_1$ as $\newi_1, \newi_2, \dotsc, \newi_k$, where $k \in \N$.
Since $w^{-1}(a_1) = \emptyset$ by assumption in this case, we can define $z \in D(\A)$ by
\[
z := h(w,y_{\newi_1},h(w,y_{\newi_2},h(w,y_{\newi_3}, \dotsc h(w,y_{\newi_k}, y_{\newi_k}) \dotsc ))).
\]
We get $f(z) = 0$ and $z(A_1) = \{q_1\}$. But $f$ agrees with an evaluation on
$\{w,z\}$. So this is a contradiction.

\medskip

\noindent \emph{Case 2: $w^{-1}(a_1) \ne \emptyset$.} We can enumerate
$A_\Sigma := w^{-1}(\Sigma) = \{\newj_1,\newj_2,\dotsc,\newj_\ell\}$, where
$w(\newj_\ell) = a_1$. Now define the map $\widehat{\phantom{a}}\, \colon A_1
\to A_1$ by
\[
\widehat \newi := \newi \cdot \newj_1\newj_2 \dotsm \newj_\ell.
\]
This map is well defined because $w(\newj_\ell) = a_1$ and so, for all $\newi
\in A_1 = w^{-1}(q_1)$, we have $w(\widehat \newi) = w(\newi) \cdot
w(\newj_1)w(\newj_2) \dotsm w(\newj_\ell) = q_1$.

Now let $\newi \in A_1$. We are supposing that $f$ is not given by evaluation
at~$\widehat \newi \in A_1$. Using the claim, there is $x_\newi \in f^{-1}(0)$
such that $0=f(x_\newi) \ne x_\newi(\widehat
\newi)$. Since
\[
0 \ne x_\newi(\widehat \newi)
= x_\newi(\newi) \cdot x_\newi(\newj_1)x_\newi(\newj_2) \dotsm x_\newi(\newj_\ell),
\]
we have $x_\newi(\newi) \in Q$ and $x_\newi(A_\Sigma) \subseteq \Sigma$. Using
Definition~\ref{def:quasi}, set $y_\newi := w \curlywedge x_\newi$. Then
$f(y_\newi) = 0$ and $y_\newi(\newi) = q_1$. We will use the order
$\sqsubseteq$ and partial join $\sqcup$ from Definition~\ref{def:order}. Since
$x_\newi(A_\Sigma) \subseteq \Sigma$ and $y_\newi = w \curlywedge x_\newi$, it
follows that $y_\newi \sqsubseteq w$.

Again enumerate $A_1$ as $\newi_1, \newi_2, \dotsc, \newi_k$. Note that the
quasi-equation
\[
u_1 \sqsubseteq v \ \& \ u_2 \sqsubseteq v \implies u_1 \sqcup u_2 \sqsubseteq v
\]
holds on $M$ and therefore on $D(\A)$.
Since we have shown that $w$ is an upper bound for $y_{\newi_1}, y_{\newi_2}, \dotsc, y_{\newi_k}$ with
respect to~$\sqsubseteq$, it follows that we can define
$z := (\dotsb(( y_{\newi_1}\sqcup y_{\newi_2}) \sqcup
y_{\newi_3}) \dotsb ) \sqcup y_{\newi_k}$ in $D(\A)$.
We have $f(z) = 0$ and $z(A_1) = \{q_1\}$. But
$f$ agrees with an evaluation on $\{w,z\}$. So this is a contradiction.
\end{proof}

\begin{corollary}\label{cor:loops}
Let\/ $\M$ be a finite automatic algebra such that every edge is a loop
\textup(that is, such that\/ $qa \in \{q,0\}$, for all $q \in Q$ and\/ $a \in
\Sigma$\textup). Then\/ $\M$ is dualizable.
\end{corollary}
\begin{proof}
We use independence of the generator (Theorem~\ref{thm:indep}). Let $q \in Q$.
Then $\{q\}$ is a component of~$\M$. By Remark~\ref{rem:easy}(3), we can assume
$q$ is not isolated. So there is at least one $a \in \Sigma$ with $qa = q$. By
Lemma~\ref{lm:quasiv} and Remark~\ref{rem:easy}(2), we can assume every letter
in $\Sigma$ fixes~$q$. So now we can assume that every letter in $\Sigma$ acts
as the identity on~$Q$. By Lemma~\ref{lm:quasiv}, we can assume $\M$ has only
one state. Thus $\M$ is dualizable by Theorem~\ref{lem:constants}.
\end{proof}

%%%%%%%%%%%%%%%%%%%%%%%%%%%%%%%%%%%%%%%%%%%%%%%%%%%%%%%%%%%%%%%%%%%%%%%%%%%%%%%%%%%%%%%%%%
%%%%%%%%%%%%%%%%%%%%%%%%%%%%%%%%%%%%%%%%%%%%%%%%%%%%%%%%%%%%%%%%%%%%%%%%%%%%%%%%%%%%%%%%%%
%%%%%%%%%%%%%%%%%%%%%%%%%%%%%%%%%%%%%%%%%%%%%%%%%%%%%%%%%%%%%%%%%%%%%%%%%%%%%%%%%%%%%%%%%%

\section{Letters acting as commuting permutations}\label{sec:trickydual}

The previous section gave a dualizability result for finite total automatic
algebras in which the range of each letter is as small as possible. In this
section we consider the opposite extreme, that is, where each letter acts as a
permutation. We are able to prove dualizability if we also assume that, on each
component, the set of permutations is a coset of a subgroup of an abelian
permutation group.

Note that, if the letters of an automatic algebra act as commuting permutations,
then the algebra satisfies the equations $z \cdot xy \approx z \cdot yx$ and
$z \cdot x^m \approx z\cdot x^n$, for some $m > n \ge 1$,
and so the algebra is finitely based by Boozer~\cite[Theorem 1.12]{B:PhD}.

\begin{definition}
Let $\M = \langle Q \cup \Sigma \cup \{0\}; \cdot \rangle$ be a finite
connected automatic algebra. We say that $\M$ is \emph{letter-affine} if
\begin{enumerate}
 \item each $a \in \Sigma$ acts as a permutation $\rho_a$ of~$Q$,
 \item the permutations in $\{\, \rho_a \mid a \in \Sigma \,\}$ commute,
     and
 \item for all $a,b,c \in \Sigma$ there exists $d \in \Sigma$ such that
     $\rho_a \circ \rho_b^{-1} \circ \rho_c = \rho_d$.
\end{enumerate}
A finite automatic algebra is \emph{letter-affine} if each of its component
subalgebras is letter-affine.
\end{definition}

The aim of this section is to prove the following.

\begin{theorem} \label{thm:affine}
Every letter-affine automatic algebra is dualizable.
\end{theorem}

As special cases, we will get the following two results.

\begin{corollary}  \label{thm:perm}
Let\/ $\M$ be a finite automatic algebra with $\Sigma=\{a\}$. If\/ $a$ acts as
a permutation of\/~$Q$, then $\M$ is dualizable.
\end{corollary}

\begin{corollary} \label{thm:dualabel}
Let\/ $\M$ be a finite automatic algebra.  If\/ $\Sigma$ acts as an abelian group of
permutations of\/~$Q$, then $\M$ is dualizable.
\end{corollary}

\begin{remark}\label{rem:general}
We can use independence of the generator to broaden the scope of
Theorem~\ref{thm:affine}. The letters in $\Sigma$ can act as partial
permutations of~$Q$ provided that, on each component of $\M$, each such partial
permutation is either totally defined or totally undefined. More precisely: a
finite automatic algebra $\M$ is dualizable if, for each component $C$ of~$\M$,
the subalgebra of $\M$ with universe $C \cup \Sigma_C \cup \{0\}$ is
letter-affine, where $\Sigma_C := \{\, a \in \Sigma \mid \dom a \cap C \ne
\emptyset\,\}$. (Use Theorem~\ref{thm:indep}, Lemma~\ref{lm:quasiv} and
Remark~\ref{rem:easy} (2), (3).)
\end{remark}

We shall say that an automatic algebra $\M$ is \emph{permutational} if every $a
\in \Sigma$ acts as a permutation of~$Q$, and that $\M$ has \emph{commuting
letters} if it satisfies the equation $x\cdot yz \approx x\cdot zy$. In
particular, every letter-affine automatic algebra is permutational and has
commuting letters.

For the remainder of this section, we consider a fixed finite automatic
algebra $\M = \langle Q \cup \Sigma \cup \{0\}; \cdot \rangle$ that is
permutational and has commuting letters. Our aim is to prove that, if $\M$ is
letter-affine, then it is dualizable. Because some parts of our argument may
have future use, we will not assume that $\M$ is letter-affine until that
assumption is needed.

Let $G_1,\dotsc,G_n$ be the components of $\M$, so that $Q = G_1 \cup \dots
\cup G_n$. We start by showing that each $G_i$ can be viewed as a finite abelian group.

\begin{claim}\label{cl:new}
For each $i \in \{1,\dotsc,n\}$, there is a binary operation $*$ on $G_i$ and
a map $-_{(i)} \colon \Sigma \to G_i$
such that
\begin{enumerate}
 \item $(G_i; *)$ is an abelian group with generating set\/~$\Sigma_{(i)}$, and
 \item for all\/ $q \in G_i$ and\/ $a \in \Sigma$, we have $q\cdot a =
     q * a_{(i)}$.
\end{enumerate}
\end{claim}
\begin{proof}
Since $\M$ is permutational, each letter $a \in \Sigma$ acts as a permutation
$\rho_{a,i}$ of~$G_i$. Define the permutation group
\[
\Pi_i := \langle \{\, \rho_{a,i} \mid a \in \Sigma \,\}\rangle \le S_{G_i}.
\]
Then $\Pi_i$ is abelian, as $\M$ has commuting letters. Note that, since $G_i$
is a component of~$\M$, the group $\Pi_i$ induces a transitive abelian group
action on~$G_i$.

Choose a state $e_i \in G_i$ and define the map $f \colon \Pi_i \to G_i$ by
$f(\phi) = \phi(e_i)$. Then $f$ is surjective, as $\Pi_i$ acts transitively
on~$G_i$. To check that $f$ is one-to-one, let $\phi, \psi \in \Pi_i$ with
$\phi(e_i) = \psi(e_i)$. Then it follows easily that $\phi = \psi$, since
$\Pi_i$ induces a transitive abelian group action on~$G_i$.

Using the bijection $f \colon \Pi_i \to G_i$, the abelian group operation
$\circ$ on $\Pi_i$ transfers to an abelian group operation $*$ on~$G_i$. Now
define the map $-_{(i)} \colon \Sigma \to G_i$ by
\[
a_{(i)} := e_i \cdot a = \rho_{a,i}(e_i) = f(\rho_{a,i}).
\]
Since $f \colon \Pi_i \to G_i$ is a group isomorphism and $\Pi_i$ is generated
by $\{\, \rho_{a,i} \mid a \in \Sigma \,\}$, it follows that $G_i$ is generated
by~$\Sigma_{(i)}$. So (1) holds.

Let $q \in G_i$ and let $a \in \Sigma$. Then $q = f(\phi) = \phi(e_i)$, for
some $\phi \in \Pi_i$. Since the permutations in $\Pi_i$ commute, we get
\begin{align*}
q \cdot a
 &= \rho_{a,i}(q)
 = \rho_{a,i} \circ \phi(e_i)
 = \phi \circ \rho_{a,i}(e_i)\\
 &= f(\phi \circ \rho_{a,i})
 = f(\phi) * f(\rho_{a,i})
 = q * a_{(i)}.
\end{align*}
So (2) holds.
\end{proof}

From now on, we use multiplicative notation for the groups $G_1, \dots, G_n$.

\begin{definition}\label{df:eH}
For each $i \in \{1,\dots,n\}$, let $e_i$ denote the identity element of the
group~$G_i$. Define the subgroup $H_i$ of $G_i$ by
\[
H_i := \langle \{g^{-1}h \mid g,h \in \Sigma_{(i)}\} \rangle \le G_i.
\]
Then $\Sigma_{(i)}$ is contained in a coset of~$H_i$. So, as $\Sigma_{(i)}$ is
a generating set for $G_i$, the group $G_i/H_i$ is cyclic.
\end{definition}

\begin{claim}\label{cl:SigmaCoset}
The automatic algebra $\M$ is letter-affine if and only if\/ $\Sigma_{(i)}$ is
a coset of\/~$H_i$ in~$G_i$, for each $i \in \{1,\dots,n\}$.
\end{claim}
\begin{proof}
A subset $S$ of $G_i$ is a coset of a subgroup of $G_i$ if and only if $S$ is
closed under the Mal'cev operation $p(x,y,z) = xy^{-1}z$.
By Claim~\ref{cl:new}, each letter $a \in \Sigma$ acts on the group~$G_i$ as right
multiplication by $a_{(i)}$.
So the claim now follows easily.
\end{proof}

We next introduce some helpful compatible operations on~$\M$.

\begin{definition}\label{df:TK}\
\begin{enumerate}
\item For $i \in \{1,\dots,n\}$ and $g \in G_i$, the  compatible
    unary operation $\lambda_g$ on $\M$ is given by
\[
\lambda_g(v) =
\begin{cases}
gv & \mbox{if $v \in G_i$,}\\
v & \mbox{otherwise.}
\end{cases}
\]
\item The   compatible binary partial operation $\K$ on $\M$ with domain
    $\big(\bigcup_{i=1}^n G_i^2\big) \cup \Sigma^2 \cup \{(0,0)\}$ is given
    by
\[
u\K v =
\begin{cases}
0 & \mbox{if $u,v \in G_i$ with $u^{-1}v \notin H_i$, for some $i \in \{1,\dots,n\}$,}\\
u & \mbox{otherwise.}
\end{cases}
\]
\end{enumerate}
\end{definition}

We now begin an argument which will ultimately prove that, if $\M$ is
letter-affine, then it is dualizable. Consider a finite algebra $\A \in
\ISP\M$. Define $D(\A) := \hom(\A,\M)$ and assume that $f \colon D(\A) \to M$
is $\max(4,2n+1)$-locally an evaluation. We aim to prove that $f$ is an
evaluation.

If $\ran f = \{0\}$, then $f$ is given by evaluation at~$0^\A$.
Using Corollary~\ref{cor:gst}, we can now assume that
$\ran f \subseteq Q \cup \{0\}$, with $\ran f \cap Q \ne \emptyset$. By re-indexing the components, we
can assume that $\ran f \cap G_i \ne\varnothing$, for $i \in
\{1,\dotsc,\ell\}$, and $\ran f \cap G_i =\varnothing$, for $i \in \{\ell + 1,
\dotsc, n\}$.

We shall use the quasi-order $\preccurlyeq$ on $M$ given by
Definition~\ref{def:quasi}.

\begin{claim}\label{cl:lambda}
The set\/ $f^{-1}(Q)$ is a `principal filter' of\/ $D(\A)$ under the
quasi-order~$\preccurlyeq$. More precisely, there exists $w \in D(\A)$ such
that, for all $x \in D(\A)$, we have $f(x) \in Q$ if and only if\/ $w \preccurlyeq x$.
Furthermore, for each $i \in \{1,\dotsc,\ell\}$, there exists $w_i \in
f^{-1}(e_i)$ with\/ $w \preccurlyeq w_i \preccurlyeq w$.
\end{claim}
\begin{proof}
The quasi-meet operation $\curlywedge \colon \M^2 \to \M$ from
Definition~\ref{def:quasi} is a homomorphism. So $D(\A) \subseteq M^A$ is
closed under~$\curlywedge$ and the map $f \colon D(\A) \to M$
preserves~$\curlywedge$ (as $f$ is $3$-locally an evaluation).

Let $x,y \in f^{-1}(Q)$. Say that $f(x) = q$ and $f(y) = r$. Then
\[
f(x \curlywedge y) = f(x) \curlywedge f(y) = q \curlywedge r = q \in Q.
\]
Thus $f^{-1}(Q)$ is also closed under~$\curlywedge$. Since $D(\A)$ is finite,
we can use $\curlywedge$ repeatedly to obtain a `least' element $w$ of
$f^{-1}(Q)$. It follows that $f(x) \in Q$ implies $w \preccurlyeq x$, for all
$x \in D(\A)$. Now assume that $x \in D(\A)$ with $w \preccurlyeq x$. Then $w
\curlywedge x = w$ and so $f(w) \curlywedge f(x) = f(w) \in Q$. This implies
that $f(x) \in Q$.

Fix $i \in \{1,\dotsc,\ell\}$ and choose $x_i \in f^{-1}(G_i)$. Say that
$f(x_i) = g \in G_i$. Then $y_i := \lambda_{g^{-1}} (x_i) \in D(\A)$ with
$f(y_i) = \lambda_{g^{-1}}(g) = e_i$, as $f$ preserves~$\lambda_{g^{-1}}$. Now
define $w_i:= y_i \curlywedge w$. Then $w_i \preccurlyeq w$ and $f(w_i) = f(y_i
\curlywedge w) = e_i \curlywedge f(w) = e_i \in Q$. So $w\preccurlyeq w_i$, by
the construction of~$w$.
\end{proof}

The homomorphism $w \colon \A \to \M$ from the claim above partitions the set $A$ into
three subsets:
\[
A_Q := w^{-1}(Q), \quad A_\Sigma := w^{-1}(\Sigma),\quad \text{and} \quad  A_0 := w^{-1}(0).
\]
If $f$ is an evaluation, then it must be given by evaluation at an element
of~$A_Q$, as $f(w) \in Q$. Since $w$ is a homomorphism and $\M$ is a total automatic algebra, it is
easy to see that $A_Q \cdot A_\Sigma \subseteq A_Q$ in~$\A$, and that all other
products in $\A$ belong to~$A_0$.

\begin{claim}\label{oneorb}
The set $A_Q$ is connected by $A_\Sigma$ in the following
sense:
\begin{quote}
For all\/ $\newi,\newip \in A_Q$, we have
     $\newi\cdot
\newj_1 \newj_2\dotsm \newj_j = \newip\cdot
\newjp_1 \newjp_2 \dotsm \newjp_k$ in~$\A$, for some $j,k \ge
0$ and some
$\newj_1,\newj_2,\dotsc,\newj_j,\newjp_1,\newjp_2,\dots,\newjp_k\in
A_\Sigma$.
\end{quote}
\end{claim}
\begin{proof}
Define $\newi \equiv \newip$ to mean that the above relation holds. Then
$\equiv$ is an equivalence relation on~$A_Q$, as $\M$ satisfies $x\cdot yz
\approx x \cdot zy$. Suppose, by way of contradiction, that $\equiv$ is not the
total relation on~$A_Q$. Then we can partition $A_Q$ as $B \cup C$, where $B,C
\ne \emptyset$ and $B \cap C = \emptyset$, such that each of $B,C$ is a union
of $\equiv$-classes. It follows that $B \cdot A_\Sigma \subseteq B$ and $C
\cdot A_\Sigma \subseteq C$ in~$\A$.

We can now define $x, y \in D(\A)$ by
\[
x(v) =
\begin{cases}
w(v) & \text{if $v \notin B$,}\\
0 & \text{otherwise,}
\end{cases}
\quad\text{and}\quad
y(v) =
\begin{cases}
w(v) & \text{if $v \notin C$,}\\
0 & \text{otherwise.}
\end{cases}
\]
By Claim~\ref{cl:lambda}, we have $f(x)= 0 = f(y)$. So $f$ does not agree with
an evaluation on the subset $\{w,x,y\}$ of~$D(\A)$, which is a contradiction.
\end{proof}

\begin{claim} \label{cl:inGi}
Let $x \in D(\A)$ such that $f(x) \in G_i$, for some $i \in \{1,\dotsc,\ell\}$.
Then $x(A_Q) \subseteq G_i$.
\end{claim}
\begin{proof}
By Claim~\ref{cl:lambda}, we have $w \preccurlyeq x$ and therefore $x(A_Q)
\subseteq Q$ and $x(A_\Sigma) \subseteq \Sigma$. Since $f$ is $2$-locally an
evaluation, it agrees with an evaluation on $\{w,x\}$. So there exists $\newi
\in A_Q$ such that $x(\newi) = f(x) \in G_i$.

Now let $\newip \in A_Q$. By Claim~\ref{oneorb}, we have $\newi\cdot \newj_1
\newj_2\dotsm \newj_j = \newip\cdot \newjp_1 \newjp_2 \dotsm \newjp_k$
 in $\A$,
for some $j,k \ge 0$ and some
$\newj_1,\newj_2,\dotsc,\newj_j,\newjp_1,\newjp_2,\dots,\newjp_k\in A_\Sigma$.
So
\[
x(\newi)\cdot x(\newj_1) x(\newj_2)\dotsm x(\newj_j)
 = x(\newip)\cdot x(\newjp_1) x(\newjp_2) \dotsm x(\newjp_k) \text{ in } \M.
\]
Since $x(\newi), x(\newip) \in x(A_Q) \subseteq Q$ and $x(A_\Sigma) \subseteq
\Sigma$, the states $x(\newi)$ and $x(\newip)$ must belong to the same
connected component of~$\M$. Hence $x(\newip) \in G_i$.
\end{proof}

\begin{claim}
If $A_\Sigma = \emptyset$, then\/ $f$ is an evaluation.
\end{claim}
\begin{proof}
Assume $A_\Sigma = \emptyset$. Then Claim~\ref{oneorb} gives $\abs {A_Q} = 1$.
Say that $A_Q = \{\newi\}$. We will check that $f$ is given by evaluation
at~$\newi$. Let $x \in D(\A)$. Then $f$ agrees with an evaluation on $\{w,x\}$.
But this must be evaluation at~$\newi$, since we have $f(w) \in Q$ and
$w^{-1}(Q) = A_Q = \{\newi\}$.
\end{proof}

By the previous claim, we can assume that $A_\Sigma \ne \emptyset$. Enumerate
the set $A_\Sigma$ as $\gamma_1, \gamma_2, \dotsc, \gamma_\kappa$. Since the
groups $G_1,\dotsc,G_n$ are finite, we can choose $m \in \N$ so that these
groups all have exponent dividing~$m$ (that is, they all satisfy the equation $x^m
\approx e$).
Now define the map $\widehat{\phantom{a}}\, \colon A_Q \to A_Q$ by
\[
\widehat \newi := \newi\cdot (\gamma_1)^m \dotsm (\gamma_\kappa)^m
\]
and define $\widehat A_Q := \{\, \widehat \newi \mid \newi \in A_Q \,\}$.

\begin{claim} \label{cl:hat}\
\begin{enumerate}
 \item We have $\widehat A_Q \cdot A_\Sigma \subseteq  \widehat A_Q$ in
     $\A$.
 \item Let\/ $X \subseteq D(\A)$ and let\/ $\newi \in A_Q$. If\/ $f \rest
     X$ agrees with evaluation at\/~$\newi$, then $f \rest X$ also agrees
     with evaluation at\/ $\widehat\newi$.
\end{enumerate}
\end{claim}
\begin{proof}
Part (1) follows because $A_Q \cdot A_\Sigma \subseteq A_Q$ and $\M$ satisfies
$x \cdot yz \approx x \cdot zy$. For part (2), assume $f \rest X$ agrees with
evaluation at~$\newi$ and let $x \in X$. First assume  $f(x) = 0$. Then
$x(\newi) = f(x) = 0$ and it follows easily that $x(\widehat\newi) = 0 = f(x)$.
Now assume $f(x) \ne 0$. Then $w \preccurlyeq x$. So $x(\newi) \in Q$ and
$x(A_\Sigma) \subseteq \Sigma$. Say that $x(\newi) \in G_i$. Since the exponent of $G_i$
divides~$m$, it follows by Claim~\ref{cl:new}(2) that $x(\widehat\newi) =
x(\newi) = f(x)$.
\end{proof}

\begin{claim}\label{trans2}
The set $A_\Sigma$ acts transitively on $\widehat A_Q$ in the
following sense:
\begin{quote}
For all\/ $\newi,\newip \in \widehat A_Q$, we have
     $\newi\cdot
\newj_1\newj_2\dotsm \newj_k = \newip$ in~$\A$, for some $k \ge
0$ and some $\newj_1,\newj_2,\dotsc,\newj_k \in A_\Sigma$.
\end{quote}
\end{claim}
\begin{proof}
Let $\newi,\newip \in \widehat A_Q$. As $\M$ satisfies $x \cdot y^{2m} \approx
x \cdot y^m$, it follows that $\newip \cdot \alpha^m = \newip$, for all $\alpha
\in A_\Sigma$. By Claim~\ref{oneorb}, we have $\newi\cdot
\newj_1 \newj_2\dotsm \newj_j = \newip\cdot \newjp_1 \newjp_2 \dotsm \newjp_k$ in~$\A$.
So $\newi\cdot
\newj_1 \newj_2\dotsm \newj_j (\newjp_1 \newjp_2 \dotsm \newjp_k)^{m-1}
= \newip \cdot (\newjp_1 \newjp_2 \dotsm \newjp_k)^m = \newip$
in~$\A$, as required.
\end{proof}

Now define the subset $\Izero$ of $\widehat A_Q$ by
\[
\Izero := \set{\newi \in \widehat A_Q}{$(\exists x \in
f^{-1}(0))\ x(\newi)\ne 0$}.
\]
Note that $f$ cannot be evaluation at any element of~$\Izero$. We next
construct a single homomorphism $z_\Izero \in D(\A)$ to witness this fact.

\begin{claim}\label{cl:nonzero2}
There exists $z_\Izero \in D(\A)$ such that $f(z_\Izero) = 0$ and
$z_\Izero(\Izero) \subseteq Q$.
\end{claim}
\begin{proof}
We use the order $\sqsubseteq$ and associated partial join $\sqcup$ from
Definition~\ref{def:order}. Note that $D(\A)$ is closed under~$\sqcup$ and that
$f$ preserves~$\sqcup$.

Fix $\newi \in \Izero$. We first show that there exists $z_\newi \in f^{-1}(0)$
with $z_\newi(\newi) \in Q$ and $z_\newi \sqsubseteq w$. By the definition of
$\Izero$, there exists $x_\newi \in f^{-1}(0)$ with $x_\newi(\newi)\ne 0$.
Since $\newi \in \widehat A_Q$, it follows easily that $x_\newi(\newi) \in Q$
and $x_\newi(A_\Sigma) \subseteq \Sigma$. Thus, if we put $z_\newi := w
\curlywedge x_\newi$, then $z_\newi(\newi)\in Q$ and $z_\newi \sqsubseteq w$.
Finally, since $f$ preserves~$\curlywedge$, we get $f(z_\newi) = f(w
\curlywedge x_\newi) = f(w) \curlywedge f(x_\newi) = f(w) \curlywedge 0 = 0$.

Now enumerate $\Izero = \{\newi_1,\dots,\newi_k\}$.  Because $z_{\sigma_i}
\sqsubseteq w$, for all $i \in \{1,\dotsc,k\}$, we can define $z_\Izero :=
(\dotsb(( z_{\newi_1}\sqcup z_{\newi_2}) \sqcup z_{\newi_3}) \dotsb ) \sqcup
z_{\newi_k}$ in $D(\A)$. Then $f(z_\Izero)=0$ since $f$ preserves~$\sqcup$, and
$z_\Izero(\Izero) \subseteq Q$ by construction.
\end{proof}

\begin{definition}
Using the homomorphisms $w_i$ from Claim~\ref{cl:lambda} and $z_\Izero$ from
Claim~\ref{cl:nonzero2}, we define the subset $C$ of $A_Q$ by
\[
\Iast := \widehat{A}_Q \cap \Big(\bigcap_{i=1}^\ell w_i^{-1}(e_i)\Big) \cap z_\Izero^{-1}(0).
\]
For $i \in \{1,\dotsc,\ell\}$, define
\[
Y_i = \set{y \in D(\A)}{$f(y)=e_i$ and $w \preccurlyeq y \preccurlyeq w$}.
\]
Let $Y := Y_1 \cup \dots \cup Y_\ell$.
\end{definition}

\begin{claim}\label{cl:eval}
If\/ $f\rest Y$ agrees with evaluation at some $\newi \in
\Iast$, then $f$ is an evaluation.
\end{claim}
\begin{proof}
Assume $f\rest Y$ is given by evaluation at~$\newi$, for some
$\newi \in \Iast$. Let $x \in D(\A)$. We will check that $f(x)
= x(\newi)$.

\smallskip

\emph{Case 1:} $f(x) = 0$. Since $\newi \in \Iast \subseteq z_\Izero^{-1}(0)$,
we have $\newi \notin \Izero$, by Claim~\ref{cl:nonzero2}. Since $\newi \in
\Iast \subseteq \widehat A_Q$, the definition of $\Izero$ ensures that
$x(\newi) = 0 = f(x)$.

\smallskip

\emph{Case 2:} $f(x) \in G_i$, for some $i \in \{1,\dotsc,\ell\}$. Say that
$f(x) = g \in G_i$. Define $y := \lambda_{g^{-1}}(x)\curlywedge w \preccurlyeq
w$. Then $f(y) = e_i \curlywedge f(w) = e_i$ and so $w \preccurlyeq y$. Thus $y
\in Y_i$, giving $e_i = f(y) = y(\newi)$. By Claim~\ref{cl:inGi}, we have
$x(\sigma) \in G_i$. Therefore
\[
e_i = y(\newi) = \lambda_{g^{-1}}(x(\sigma))\curlywedge w(\sigma) = g^{-1}x(\sigma),
\]
and so $f(x) = g = x(\newi)$, as required.
\end{proof}

\begin{claim}\label{cl:smallX}
Let\/ $X \subseteq D(\A)$ with $\abs X \le n$. Then there exists $\newi \in
\Iast$ such that $f\rest X$ agrees with evaluation at~$\newi$.
\end{claim}
\begin{proof}
As $f$ is $(2n+1)$-locally an evaluation, there is $\newip\in A$ such that $f$
agrees with evaluation at $\newip$ on $X' := X \cup \{w_1,\dots,w_\ell,
z_\Izero\}$. Since $w_1(\newip) = f(w_1) = e_1 \in G_1$, we have $\newip \in
A_Q$. So $f$ also agrees with evaluation at $\newi := \widehat\newip$ on $X'$, by
Claim~\ref{cl:hat}. Because $e_i = f(w_i)=w_i(\newi)$, we get
$\newi\in w_i^{-1} (e_i)$. Because $0 = f(z_\Izero) =
z_\Izero(\newi)$, we get $\newi \in z_\Izero^{-1}(0)$. Thus
$\newi \in \Iast$.
\end{proof}

\begin{claim} \label{cl:inHi}
For all $i \in \{1,\dotsc,\ell\}$ and $y \in Y_i$, we have $y(\Iast) \subseteq
H_i$.
\end{claim}
\begin{proof}
Fix $i \in \{1,\dotsc,\ell\}$ and $y \in Y_i$. Let $\newi \in \Iast$. We shall
use the  binary partial operation $\K$ from Definition~\ref{df:TK}. Since $y, w_i \in
Y_i$, we have $f(y), f(w_i) \in G_i$ and $w \preccurlyeq y, w_i \preccurlyeq w$.
It follows from Claim~\ref{cl:inGi}
that $(y, w_i) \in \dom {\K}$ in
$D(\A)$. So we can define $x := y \K w_i \in D(\A)$ with
\[
f(x) = f(y \K w_i) = f(y) \K f(w_i) = e_i \K e_i = e_i.
\]
As $\newi \in A_Q$, this implies that $x(\newi) \in G_i$, using
Claim~\ref{cl:inGi} again. As $\newi \in \Iast$, we have $w_i(\newi) = e_i$.
Therefore
\[
y(\newi) \K e_i = y(\newi) \K w_i(\newi) = x(\newi) \in G_i,
\]
whence $y(\newi) \in H_i$.
\end{proof}

We remark in passing that at this point we have already accumulated enough
information to prove Corollary~\ref{thm:perm}. (Assume $|\Sigma|=1$. For all $i
\in \{1,\dotsc,\ell\}$, we have $\abs{H_i}=1$ and so $y(\Iast)= \{e_i\}$, for
all $y \in Y_i$, by Claims~\ref{cl:smallX} and~\ref{cl:inHi}. Thus $f \rest Y$
agrees with evaluation at any $\sigma \in \Iast$, whence $f$ is an evaluation
by Claim~\ref{cl:eval}.)

\goodbreak

\begin{definition}\
\begin{enumerate}
\item Let $\scrM$ denote the $Y\times \Iast$ matrix over $Q$ whose entry at
    position $(y,\newi)$ is $y(\newi)$. For $i \in \{1,\dotsc,\ell\}$ and
    $y \in Y_i$, the row of $\scrM$ at position $y$ is $y\rest \Iast \in
    H_i^{\Iast}$, by Claim~\ref{cl:inHi}. For each $\newi \in \Iast$, the
    column of $\scrM$ at position $\newi$ belongs to $H_1^{Y_1} \times
    \cdots \times H_\ell^{Y_\ell}$.
\item Partition $\scrM$ via $Y=Y_1\cup\cdots\cup Y_\ell$, and let $\scrM_i$
    denote the corresponding $Y_i\times \Iast$ submatrix of~$\scrM$ with
    entries in $H_i$.
\end{enumerate}
\end{definition}

\begin{claim} \label{cl:cols-coset}
The columns of\/ $\scrM$ form a coset of a subgroup of $H_1^{Y_1}
     \times \cdots \times H_\ell^{Y_\ell}$.
\end{claim}
\begin{proof}
It suffices to prove that the set of columns is closed under the Mal'cev
operation of $H_1^{Y_1} \times \cdots \times H_\ell^{Y_\ell}$. So let
$\newi,\newip,\newiq \in \Iast$ and let $\col\newi$, $\col\newip$, $\col\newiq$
denote the associated columns of $\scrM$. We want to find $\newir \in \Iast$
such that $\col\newi\col\newip^{-1}\col\newiq = \col\newir$, computed in
$H_1^{Y_1}\times \cdots \times H_\ell^{Y_\ell}$.

By Claim~\ref{trans2}, we can find $\newj_1,\dots,\newj_k \in A_\Sigma$ such
that $\newi = \newip\cdot \newj_1\cdots
\newj_k$ in~$\A$.  Define $\newir :=
\newiq\cdot\newj_1\cdots \newj_k \in A$.  We will first show that $\theta \in \Iast$.
Clearly $\theta \in \widehat{A}_Q$, by Claim~\ref{cl:hat}(1), and
$z_\Izero(\newir) = z_\Izero(\newiq)\cdot z_\Izero(\newj_1)\cdots
z_\Izero(\newj_k) = 0$, as $\rho \in z_\Izero^{-1}(0)$. So we just need to
check that $w_i(\newir)=e_i$, for $i \in \{1,\dotsc,\ell\}$.

Let $i \in \{1,\dotsc,\ell\}$ and define $a_j := w_i(\newj_j) \in \Sigma$, for each $j \in \{1,\dotsc,k\}$.  Then
\[
w_i(\newir)
 = w_i(\newiq)\cdot a_1\dotsb a_k
 = e_i\cdot a_1 \dotsb a_k = w_i(\newip)\cdot a_1 \dotsb a_k = w_i(\newi) = e_i.
\]
Thus $\newir \in \Iast$.

Now let $y \in Y_i$, for some $i \in \{1,\dotsc,\ell\}$. It remains to check
that we have $y(\newi)y(\newip)^{-1}y(\newiq) = y(\newir)$ in the group~$G_i$.
Recall from Claim~\ref{cl:new} that the map $-_{(i)} \colon \Sigma \to G_i$
satisfies $g \cdot a = ga_{(i)}$, for all $g \in G_i$ and $a \in \Sigma$. (The
left side is evaluated in the automatic algebra $\M$ and the right side in the
group $G_i$.) For each $j \in \{1,\dotsc,k\}$, define $b_j := y(\newj_j) \in
\Sigma$ with $b_{j(i)} := (b_j)_{(i)} \in \Sigma_{(i)} \subseteq G_i$.
Calculating in the abelian group $G_i$, we get
\begin{align*}
y(\newi)y(\newip)^{-1}y(\newiq)
 &= y(\newip\cdot \newj_1\cdots \newj_k) y(\newip)^{-1}y(\newiq)\\
 &= y(\newip)b_{1(i)} \dotsb b_{k(i)}y(\newip)^{-1}y(\newiq)\\
 &= y(\newiq)b_{1(i)}\dotsb b_{k(i)} = y(\rho \cdot \newj_1\cdots \newj_k) = y(\newir),
\end{align*}
as required.
\end{proof}

Note that the foregoing analysis assumed only that $\M$ is permutational with
commuting letters. At this point we introduce the further assumption that $\M$
is letter-affine. So $\Sigma_{(i)}$ is a coset of $H_i$ in $G_i$, by
Claim~\ref{cl:SigmaCoset}. This means that $\Sigma_{(i)}$ is closed under the
Mal'cev operation $p_i$ on $G_i$ given by $p_i(x,y,z) = xy^{-1}z$. The next
general lemma shows that $p_i$ extends to a compatible partial operation
on~$\M$.

\begin{claim}\label{cl:extends}
Let\/ $i \in \{1,\dotsc,n\}$ and assume that\/ $\phi \colon (G_i)^k \to G_i$ is
a group homomorphism with $\phi\big((\Sigma_{(i)})^k\big) \subseteq
\Sigma_{(i)}$. Then $\phi$ extends to a compatible partial operation $\psi
\colon (G_i)^k \cup \Sigma^k \cup \{\hat 0\} \to M$ on~$\M$ with
$\psi(\Sigma^k) \subseteq \Sigma$.
\end{claim}
\begin{proof}
For $a_1,\dotsc,a_k \in \Sigma$, choose $\psi(a_1,\dotsc,a_k)$ to be any $b \in
\Sigma$ such that $b_{(i)} = \phi((a_1)_{(i)},\dotsc,(a_k)_{(i)})$.
\end{proof}

We shall also use this claim to extend certain group endomorphisms of $H_i$
to compatible partial operations on~$\M$. To this end, define $n_i := |G_i /
H_i|$, pick a fixed element $a_i \in \Sigma_{(i)}$ and define $u_i := a_i^{n_i}
\in H_i$.

\begin{claim}\label{cl:endextends}
Let\/ $\phi \in \End {H_i}$ with $\phi(u_i) = u_i$. Then $\phi$ extends to a
compatible partial operation $\psi \colon G_i \cup \Sigma \cup \{\hat 0\} \to
M$ on~$\M$ with $\psi(G_i) \subseteq G_i$ and $\psi(\Sigma) \subseteq \Sigma$.
\end{claim}
\begin{proof}
Using Claim~\ref{cl:extends}, it suffices to show that $\phi$ extends to an
endomorphism $\xi$ of $G_i$ with $\xi(\Sigma_{(i)}) \subseteq \Sigma_{(i)}$. We
shall check that we can take $\xi(g) := a_i^t\phi(h)$, where $t \in \Z$ and $h
\in H_i$ are such that $g = a_i^th$.

We observed in Definition~\ref{df:eH} that the group $G_i/H_i$ is cyclic. Since
$\Sigma_{(i)}$ is a generating set for $G_i$ and $a_i \in \Sigma_{(i)}$, it
follows that $a_iH_i$ is a generator of $G_i/H_i$. So $G_i = \bigcup_{t\in
\mathbb Z} a_i^tH_i$, where $a_i^t \in H_i$ if and only if $n_i \mid t$.

To see that $\xi$ is well defined, let $s,t \in \mathbb Z$ and $h, k \in H_i$
with $a_i^sh=a_i^tk$.  Then $a_i^{s-t}=kh^{-1}\in H_i$, so $n_i \mid (s-t)$.
Say that $s = n_iq + t$. Then $a_i^s\phi(h) = u_i^q a_i^t \phi(h) =
a_i^t\phi(u_i^q h) = a_i^t\varphi(k)$. Thus $\xi$ is well defined.

It is easy to check that $\xi$ is an endomorphism of the group~$G_i$. Since
$\Sigma_{(i)}$ is a coset of~$H_i$ in~$G_i$, we have $\Sigma_{(i)} = a_iH_i$.
It follows that $\xi(\Sigma_{(i)}) \subseteq \Sigma_{(i)}$.
\end{proof}

\begin{claim} \label{cor:groupstuff}
For each $i \in \{1,\dotsc,\ell\}$,
\begin{enumerate}
\item the rows of\/ $\scrM_i$ form a subgroup of
    $H_i^{\Iast}$, and
\item the rows of\/ $\scrM_i$ are closed under each $\varphi \in \End{H_i}$
    such that $\varphi(u_i)=u_i$.
\end{enumerate}
\end{claim}
\begin{proof}
(1): Let $y,z \in Y_i$. We want to check that $(y \rest \Iast)(z \rest \Iast)
\in Y_i \rest \Iast$, where the multiplication is computed in the group
$H_i^{\Iast}$.

As $G_i$ is abelian, the Mal'cev operation $p \colon (G_i)^3 \to G_i$ is a
group homomorphism. As $\M$ is letter-affine, we have
$p\big((\Sigma_{(i)})^3\big) \subseteq \Sigma_{(i)}$. Thus $p$ extends to a
compatible partial operation $\overline p \colon (G_i)^3 \cup \Sigma^3 \cup
\{\hat 0\} \to M$ on~$\M$ with $\overline p(\Sigma^3) \subseteq \Sigma$, by
Claim~\ref{cl:extends}.

Since $y,z \in Y_i$, it follows using Claim~\ref{cl:inGi} that $(y, w_i, z) \in
\dom{\overline p}$ in $D(\A)$. Since $f$ is $4$-locally an evaluation, it is
easy to check that  $x := \overline p(y, w_i, z) \in Y_i$. Finally, since
$w_i(\Iast) \subseteq \{e_i\}$, we get $x(\newi) = y(\newi)z(\newi)$, for all
$\newi \in \Iast$.

(2): This part follows similarly using Claim~\ref{cl:endextends}.
\end{proof}

We need the following result about finite abelian groups, whose
proof is in the appendix.

\begin{proposition} \label{HuC}
Let $H$ be a finite abelian group with exponent dividing $m$ and let $u \in H$. Then there
is a homomorphism $\chi \colon H \to \Z_m$ such that, for all $h \in H\setminus
\{e\}$, there exists $\varphi\in \End H$ with $\varphi(u)=u$ and\/
$\chi(\varphi(h)) \ne 0$.
\end{proposition}

For each $i \in \{1,\dotsc,\ell\}$, we can use this proposition to choose a
homomorphism $\chi_i \colon H_i \to \Z_m$ such that, for all $h \in
H_i\setminus \{e_i\}$, there is $\varphi\in \End {H_i}$ with $\varphi(u_i)=u_i$
and $\chi_i(\varphi(h)) \ne 0$.

\begin{definition}
Define $\barY := Y_1 \times \dots \times Y_\ell$.  Let $\scrbarM$ denote the
$\barY\times \Iast$ matrix over $\mathbb Z_m$ whose entry at position
$((y_1,\dots,y_\ell),\newi)$ is $\sum_{i=1}^\ell \chi_i(y_i(\newi))$.
\end{definition}

\begin{claim} \label{big-claim}\
\begin{enumerate}
 \item The columns of\/ $\scrbarM$ form a coset of a
     subgroup of\/ $(\mathbb Z_m)^{\barY}$.
 \item The rows of\/ $\scrbarM$ form a subgroup of\/
     $(\mathbb Z_m)^{\Iast}$.
 \item Every row of\/ $\scrbarM$ contains at least one~$0$.
 \item If\/  $\scrbarM$ has a column that is constantly~$0$, then $f$ is an
     evaluation.
\end{enumerate}
\end{claim}
\begin{proof}
(1): Choose $\newi,\newip,\newiq \in \Iast$. Let $\Cbar_\newi,
\Cbar_\newip,\Cbar_\newiq$ be the associated columns of~$\scrbarM$, and let
$\col\newi,\col\newip,\col\newiq$ be the associated columns of~$\scrM$. By
Claim~\ref{cl:cols-coset} there exists $\newir\in\Iast$ such that $\col\newi
\col\newip^{-1} \col\newiq = \col\newir$; using the fact that each $\chi_i$ is
a group homomorphism, it is easy to show that $\Cbar_\newi -\Cbar_\newip
+\Cbar_\newiq = \Cbar_\newir$, which suffices.

(2): This part follows from Claim~\ref{cor:groupstuff}(1).

(3): Let $(y_1,\dotsc,y_\ell) \in \barY$. By Claim~\ref{cl:smallX}, there is
$\newi \in \Iast$ such that $y_i(\newi) = e_i$, for all $i \in
\{1,\dotsc,\ell\}$. So the row at $(y_1,\dotsc,y_\ell)$ has a $0$ in the
$\newi$ position.

(4): Assume that the $\newi$-column of $\scrbarM$ is constantly~$0$, for some
$\newi \in \Iast$.  Now let $j \in \{1,\dots,\ell\}$ and $y \in Y_j$. We shall
check that $y(\newi) = e_j$. It will then follow by Claim~\ref{cl:eval} that
$f$ is an evaluation.

Let $\varphi\in \End {H_j}$ with $\varphi(u_j)=u_j$. It suffices to show that
$\chi_j(\varphi(y(\newi))) = 0$. By Claim~\ref{cor:groupstuff}(2), there is
some $z \in Y_j$ such that $z \rest \Iast = \varphi(y \rest \Iast)$. Now
consider $(w_1,\dots,w_{j-1},z,w_{j+1},\dots,w_\ell) \in \barY$. As the
$\newi$-column of $\scrbarM$ is constantly~$0$, we get
\[
0 = \chi_j(z(\newi)) + \sum_{i\ne j}\chi_i(w_i(\newi))
= \chi_j(z(\newi)) + \sum_{i\ne j}\chi_i(e_i) = \chi_j(z(\newi))
\]
and so $\chi_j(\varphi(y(\newi))) = \chi_j(z(\newi)) = 0$, as required.
\end{proof}

The proof of the next result is in the appendix.

\begin{proposition} \label{no-matrix}
Assume $M$ is a $j \times k$ matrix over $\mathbb Z_m$ whose rows form a
subgroup of $(\mathbb Z_m)^k$, whose columns form a coset of a subgroup of
$(\mathbb Z_m)^j$, and which is such that every row contains at least one~$0$.
Then some column is constantly~$0$.
\end{proposition}

Using this proposition and Claim~\ref{big-claim}, it follows that $f$ is an
evaluation. Hence we have proved that $\M$ is dualizable if it is
letter-affine.

%%%%%%%%%%%%%%%%%%%%%%%%%%%%
%%%%%%%%%%%%%%%%%%%%%%%%%%%%

\section{Two classification results}\label{sec:class}

In this section, we characterize dualizability within two special classes of
finite automatic algebras: $\abs \Sigma = 1$ and $\abs Q = 2$.

Recall that the term `whiskery cycles' was introduced in
Definition~\ref{def:whiskcyc}.

\begin{lemma}\label{lem:onewhiskcyc}
Let\/ $\M$ be a finite automatic algebra with $\Sigma = \{a\}$.
If the letter $a$ acts as whiskery cycles, then\/ $\M$ is
dualizable.
\end{lemma}
\begin{proof}
Assume $a$ acts as whiskery cycles. Each state of $\M$ is (1)~in an $a$-cycle,
(2)~only one step away from an $a$-cycle, or (3)~not in the domain of~$a$.
Using Remark~\ref{rem:easy}, we can assume that $\M$ has no redundant or
isolated states. But the states satisfying (2) are redundant, and the states
satisfying (3) are isolated. Thus we can assume that $a$ acts as a permutation
of~$Q$, and so $\M$ is dualizable by Corollary~\ref{thm:perm}.
\end{proof}

\begin{theorem}[Classification for $\abs\Sigma = 1$]\label{thm:class1}
Let\/ $\M = \langle Q \cup \Sigma \cup \{0\}; \cdot \rangle$ be a finite
automatic algebra with $\abs \Sigma = 1$. Then $\M$ is dualizable if and only
if the letter acts as whiskery cycles \textup(i.e., $\M$ satisfies $vxx \approx
wxx \implies vx \approx wx$\textup).
\end{theorem}
\begin{proof}
This follows directly from Theorem~\ref{thm:wc} and
Lemma~\ref{lem:onewhiskcyc}.
\end{proof}

We next complete the classification for $2$-state automatic algebras. The
following two algebras are not covered by any of the results we have proved so
far.

%%%%%%%%%%%%%%%%%%%%%%%%%%%%%%%%%%%
%%%%%%%%%%%%%%%%%%%%%%%%%%%%%%%%%%%
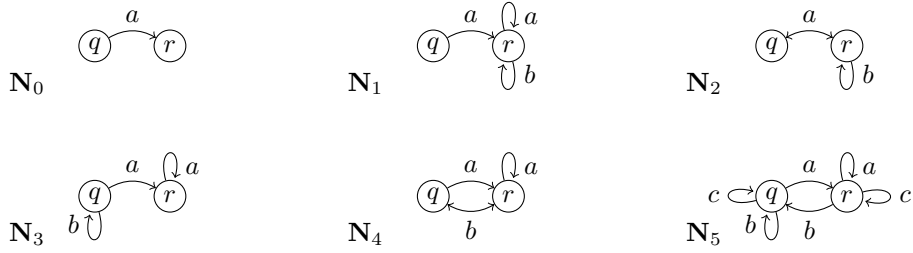
\begin{figure}
\begin{center}
\begin{tikzpicture}
%% N_1
   \begin{scope}
      \node at (-0.9,-0.5) {$\mathbf N_0$};
      \node[state] (1) at (0,0) {$q$};
      \node[state] (2) at (1,0) {$r$};
      \path (1) edge [->, bend left] node {$a$} (2);
   \end{scope}
%% N_2
   \begin{scope}[xshift=4.5cm]
      \node at (-0.9,-0.5) {$\mathbf N_1$};
      \node[state] (1) at (0,0) {$q$};
      \node[state] (2) at (1,0) {$r$};
      \path (1) edge [->, bend left] node {$a$} (2);
      \path (2) edge [loop above, very near end, right] node {$a$} (2);
      \path (2) edge [loop below, very near start, right] node {$b$} (2);
   \end{scope}
%% N_3
   \begin{scope}[xshift=9cm]
      \node at (-0.9,-0.5) {$\mathbf N_2$};
      \node[state] (1) at (0,0) {$q$};
      \node[state] (2) at (1,0) {$r$};
      \path (1) edge [<->, bend left] node {$a$} (2);
      \path (2) edge [loop below, very near start, right] node {$b$} (2);
   \end{scope}
%% N_4
   \begin{scope}[yshift=-2cm]
      \node at (-0.9,-0.5) {$\mathbf N_3$};
      \node[state] (1) at (0,0) {$q$};
      \node[state] (2) at (1,0) {$r$};
      \path (1) edge [->, bend left] node {$a$} (2);
      \path (2) edge [loop above, very near end, right] node {$a$} (2);
      \path (1) edge [loop below, very near end, left] node {$b$} (1);
   \end{scope}
%% N_5
   \begin{scope}[yshift=-2cm,xshift=4.5cm]
      \node at (-0.9,-0.5) {$\mathbf N_4$};
      \node[state] (1) at (0,0) {$q$};
      \node[state] (2) at (1,0) {$r$};
      \path (1) edge [->, bend left] node {$a$} (2);
      \path (2) edge [loop above, very near end, right] node {$a$} (2);
      \path (1) edge [<->, bend right, swap] node {$b$} (2);
   \end{scope}
%% N_6
   \begin{scope}[yshift=-2cm,xshift=9cm]
      \node at (-0.9,-0.5) {$\mathbf N_5$};
      \node[state] (1) at (0,0) {$q$};
      \node[state] (2) at (1,0) {$r$};
      \path (1) edge [->, bend left] node {$a$} (2);
      \path (2) edge [loop above, very near end, right] node {$a$} (2);
      \path (2) edge [->, bend left] node {$b$} (1);
      \path (1) edge [loop below, very near end, left] node {$b$} (1);
      \path (1) edge [loop left] node {$c$} (1);
      \path (2) edge [loop right] node {$c$} (2);
   \end{scope}
\end{tikzpicture}
\end{center}
\caption{The minimal non-dualizable $2$-state automatic algebras}\label{fig:2nd}
\end{figure}
%%%%%%%%%%%%%%%%%%%%%%%%%%%%%%%%%%%
%%%%%%%%%%%%%%%%%%%%%%%%%%%%%%%%%%%

\begin{lemma}\label{lem:2state2}
The $2$-state automatic algebra $\mathbf N_4$ from Figure~\ref{fig:2nd} is
inherently non-dualizable.
\end{lemma}
\begin{proof}
We will use Lemma~\ref{lem:IND} with the map $\mu \colon \N \to \N$ given by
$\mu(n) := n$. Define $A_0, A \subseteq M^\N$ by
\[
 A_0 := \{\, q\ov ri  \mid i \in \N \,\}, \quad
 A:=\left(Q^{\N}\comp \{\underline{q}\}\right) \cup \left( \Sigma^{\N}\comp \{\underline{b}\}\right)
    \cup \{\underline{0}\}.
\]
It is straightforward to check that $A$ is the universe of a subalgebra $\A$ of
$\M^\N$. Condition~\ref{lem:IND}(2) holds, as $g = \underline{q} \not\in A$.

To see that condition~\ref{lem:IND}(1) holds, let $n \in \N$ and let $\theta$
be a congruence on $\A$ of index at most~$n$. Consider two subsets $\{\, q\ov
rj \mid j \in J \,\}$ and $\{\, q\ov rk \mid k \in K \,\}$ of $A_0$ that are
each contained in a block of~$\theta$, where $J$ and $K$ are disjoint subsets
of $\N$ with $\abs J = \abs K = n + 1$. We want to show that $\{\, q\ov ri
\mid i \in J \cup K \,\}$ is contained in a block of~$\theta$.

As the sets $\{\, b\ov aj \mid j \in J \,\}$ and $\{\, b\ov ak \mid k \in K
\,\}$ each have $n+1$ elements, there must be distinct $i,j \in J$ and distinct
$k,\ell \in K$ such that $b\ov ai \equiv_{\theta} b \ov aj$ and $b\ov ak
\equiv_{\theta} b \ov a\ell$. We have
\[
q\ov rj = q\ov rk \cdot b\ov ak \cdot b\ov aj \equiv_{\theta} q\ov rk \cdot b\ov a\ell \cdot b\ov aj = q\ov rj \ov rk.
\]
A symmetric argument shows that $q\ov rk \equiv_{\theta} q\ov rj \ov rk$ and
hence $q\ov rj \equiv_{\theta} q\ov rk$. So condition~\ref{lem:IND}(1) holds
and $\M$ is inherently non-dualizable.
\end{proof}

\begin{lemma}\label{lem:2state3}
The $2$-state automatic algebra $\mathbf N_5$ from Figure~\ref{fig:2nd} is
inherently non-dualizable.
\end{lemma}
\begin{proof}
We use Lemma~\ref{lem:IND} with $\mu \colon \N \to \N$ given by $\mu(n) := 1$.
Let $A_0, A \subseteq M^\N$ be
\[
A_0 := \{\, q\ov ri \mid i \in \N \,\}, \quad A:=\left(Q^{\N}\comp \{\underline{q}\}\right) \cup \left( \Sigma^{\N}\comp \{b,c\}^{\N}\right) \cup \{\underline{0}\}.
\]
Note that $A$ is the universe of a subalgebra $\A$ of $\M^\N$, and that
condition~\ref{lem:IND}(2) holds, as $g = \underline{q} \not\in A$.

For condition~\ref{lem:IND}(1),  let $\theta$ be a congruence on $\A$. We need
to show that $\theta\rest{A_0}$ has a unique non-trivial block. So assume that
$q \ov ri \equiv_{\theta} q\ov rj$ and $q \ov rk \equiv_{\theta} q\ov r\ell$,
for distinct $i,j,k,\ell \in \N$. It suffices to show that  $q \ov ri
\equiv_{\theta} q\ov rk$, which follows as
\[
q \ov rk = q \ov rj \cdot b \ov ci \ov ak \equiv_{\theta} q \ov ri \cdot b \ov ci \ov ak = q \ov ri \ov rk
\]
and, by symmetry, $q \ov ri \equiv_{\theta} q \ov ri \ov rk$.
\end{proof}

\goodbreak

\begin{theorem}[Classification for $\abs Q = 2$]\label{thm:class2}
Let\/ $\M = \langle Q \cup \Sigma \cup \{0\}; \cdot \rangle$ be a finite
automatic algebra with $\abs Q = 2$. Then the following are equivalent:
\begin{enumerate}
 \item $\M$ is dualizable;
 \item $\M$ satisfies the equations $xy \approx xyyy$ and $wxyz \approx
     wyxz$;
 \item none of the six automatic algebras in Figure~\ref{fig:2nd} embeds
     into~$\M$.
\end{enumerate}
\end{theorem}
\begin{proof}
(1)\,$\Rightarrow$\,(3): The six algebras in Figure~\ref{fig:2nd} are
inherently non-dualizable by Theorem~\ref{thm:wc}, Example~\ref{ex:rankill},
and Lemmas~\ref{lem:2state2} and~\ref{lem:2state3}.

(2)\,$\Rightarrow$\,(3): The algebra $\mathbf N_0$ fails the first equation, as
$qa = r \ne 0 = qaaa$. The other algebras fail the second equation, as shown
below.
\begin{align*}
 \mathbf N_1, \mathbf N_2 &\colon \ qabb = r \ne 0 = qbab; & \mathbf N_3 &\colon \ qbaa = r \ne 0 = qaba;\\
 \mathbf N_4 &\colon \ qbab = q \ne r = qabb; & \mathbf N_5 &\colon \ qabc = q \ne r = qbac.
\end{align*}

(3)\,$\Rightarrow$\,(1)\,\&\,(2): Assume $\mathbf N_i \not\hookrightarrow \M$,
for all $i \in \{0, 1, \dotsc, 5\}$. We want to show that $\M$ is dualizable
and satisfies $xy \approx xyyy$ and $wxyz \approx wyxz$. By
Remark~\ref{rem:easy}, the automatic algebra $\M$ generates the same
quasi-variety (and therefore variety) as one with no `repeated letters' and no
`totally undefined letters'. So we can assume $\M$ has no such letters, by
Theorem~\ref{thm:indep}.

Since $\mathbf N_0 \not\hookrightarrow \M$, each letter in $\Sigma$ acts as
\begin{itemize}
 \item the transposition,
 \item a constant, or
 \item a restriction of the identity.
\end{itemize}
This implies that $\M$ satisfies the first equation $xy \approx xyyy$.

First assume that each edge in the partial automaton of~$\M$ is a loop. So $\M$ is
dualizable, by Corollary~\ref{cor:loops}. Let $a,b,c \in \Sigma$ and define
$D:= \dom a \cap \dom b \cap \dom c \subseteq Q$. Let $q \in Q$. Since every
edge is a loop, if $q \in D$, then $qabc = q = qbac$, and if $q \in Q\comp D$,
then $qabc = 0 = qbac$. So $\M$ satisfies the second equation.

Now assume that there is an edge that is not a loop. Then $\Sigma$ must contain
a letter that acts as the transposition or as a constant.

First consider the case where a letter in $\Sigma$ acts as the transposition
on~$Q$. Then there can be no constants (as $\mathbf N_4 \not\hookrightarrow
\M$) and there can be no proper restricted identity (as $\mathbf N_2
\not\hookrightarrow \M$). So $\M$ is isomorphic to one of the following two
algebras.
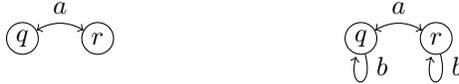
\begin{center}
\begin{tikzpicture}
   \begin{scope}
      \node[state] (1) at (0,0) {$q$};
      \node[state] (2) at (1,0) {$r$};
      \path (2) edge [<->, bend right, above] node {$a$} (1);
   \end{scope}
   \begin{scope}[xshift=4.5cm]
      \node[state] (1) at (0,0) {$q$};
      \node[state] (2) at (1,0) {$r$};
      \path (2) edge [<->, bend right, above] node {$a$} (1);
      \path (1) edge [loop below, very near start, right] node {$b$} (1);
      \path (2) edge [loop below, very near start, right] node {$b$} (2);
   \end{scope}
\end{tikzpicture}
\end{center}
In both cases, the algebra $\M$ is dualizable (by Corollaries~\ref{thm:perm}
and~\ref{thm:dualabel}) and satisfies $wxyz \approx wyxz$.

Now we are down to the case where $\Sigma$ contains a constant letter. There
can be no proper restricted identity (as $\mathbf N_1, \mathbf
N_3\not\hookrightarrow \M$) and no transposition (as $\mathbf N_4
\not\hookrightarrow \M$). As $\mathbf N_5 \not\hookrightarrow \M$, it follows
that $\M$ is isomorphic to one of the following.
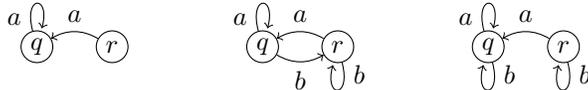
\begin{center}
\begin{tikzpicture}
  \begin{scope}
      \node[state] (1) at (0,0) {$q$};
      \node[state] (2) at (1,0) {$r$};
      \path (2) edge [->, bend right, above] node {$a$} (1);
      \path (1) edge [loop above, very near start, left] node {$a$} (1);
   \end{scope}
   \begin{scope}[xshift=3cm]
      \node[state] (1) at (0,0) {$q$};
      \node[state] (2) at (1,0) {$r$};
      \path (2) edge [->, bend right, above] node {$a$} (1);
      \path (1) edge [loop above, very near start, left] node {$a$} (1);
      \path (1) edge [->, bend right, below] node {$b$} (2);
      \path (2) edge [loop below, very near start, right] node {$b$} (2);
   \end{scope}
   \begin{scope}[xshift=6cm]
      \node[state] (1) at (0,0) {$q$};
      \node[state] (2) at (1,0) {$r$};
      \path (2) edge [->, bend right, above] node {$a$} (1);
      \path (1) edge [loop above, very near start, left] node {$a$} (1);
      \path (1) edge [loop below, very near start, right] node {$b$} (1);
      \path (2) edge [loop below, very near start, right] node {$b$} (2);
   \end{scope}
\end{tikzpicture}
\end{center}
These three algebras all satisfy $wxyz \approx wyxz$. The first two are
dualizable by Theorem~\ref{lem:constants}. So it remains to check that $\M$ is
dualizable if it is the third. In this case, define $\mathbf L$ to be the
subalgebra of $\M$ on $L := \{q, a, 0\}$. There is an embedding $\phi \colon \M
\to \mathbf L^2$ given by $x \mapsto (x,0)$, for all $x \in L$, $r \mapsto
(q,q)$ and $b \mapsto (a,a)$. Since $\mathbf L$ is dualizable by
Theorem~\ref{lem:constants}, it follows by Theorem~\ref{thm:indep} that $\M$ is
too.
\end{proof}

%%%%%%%%%%%%%%%%%%%%%%%%%%%%%%%%%%%%%%%%%%%%%%%%%%%%%%%%%%%%%%%%%%%%%%%%%%%%%%%%%%%%%%%%%%%%%%%
\section{Alternating chain}\label{sec:morenondual}

To complement Theorem~\ref{thm:affine}, we show that an
automatic algebra $\M$ can be non-dualizable if $\Sigma$ acts as
a set of commuting permutations of~$Q$. We can then give
an infinite ascending chain of automatic algebras that are
alternately dualizable and non-dualizable.

Since we are finding non-dualizable automatic algebras that are not inherently non-dualizable, we need to use the `non-inherent' version of Lemma~\ref{lem:IND}.

\begin{lemma}[Non-dualizability {\cite{three}}]\label{lem:ND}
Let\/ $\M$ be a finite algebra and let $\nu \in \N$. Assume there is a
subalgebra $\A$ of\/~$\M^I$, for some set\/~$I$, and an infinite subset\/ $A_0$
of\/ $A$ such that
\begin{enumerate}
 \item for each homomorphism $x \colon  \A \to \M$, the equivalence
     relation $\ker (x\rest{A_0})$ has a unique block of size more
     than\/~$\nu$, and
 \item the algebra $\A$ does not contain the element\/ $g$
     of\/ $M^I$ given by $g(i) := a_i(i)$, where $a_i$ is
     any element of the unique infinite block of\/ $\ker
     (\pi_i\rest{A_0})$.
\end{enumerate}
Then $\M$ is non-dualizable.
\end{lemma}

The next theorem, which is technical in its details, can be viewed as a partial converse to
 Theorem~\ref{thm:affine}. In slightly simplified terms, it states the following:
 let $\M$ be a finite permutational automatic algebra with commuting letters, at least two of which act
 differently; if $\M$ is dualizable, then it has a letter-affine subalgebra
 with at least two letters acting differently.

\begin{theorem}\label{thm:nondcomm}
Let $\M = \langle Q \cup \Sigma \cup \{0\}; \cdot \rangle$ be a finite
automatic algebra and let $m > 1$. Assume that
\begin{itemize}
 \item[(a)] each $a \in \Sigma$ acts as a permutation $\rho_a$ of\/~$Q$,
 \item[(b)] the permutations in $\{\, \rho_a \mid a \in \Sigma \,\}$ commute,
 \item[(c)] there are $b,c \in \Sigma$ such that the permutation
     $\rho_b(\rho_c)^{-1}$ of\/ $Q$ has order $m$,
 \item[(d)] for each component $C$ of\/~$\M$, there is no non-trivial subgroup
     $H$ of the symmetric group~$S_C$ such that\/ $\abs H$ divides $m$ and
     the set\/ $\{ \rho_a \rest C \mid a \in \Sigma \,\}$ contains a coset
     of\/~$H$.
\end{itemize}
Then $\M$ is non-dualizable.
%Let $\M = \langle Q \cup \Sigma \cup \{0\}; \cdot
%\rangle$ be a finite automatic algebra, in which all letters act as distinct,
%commuting, and total permutations. For distinct $a,b \in \Sigma$ let $h$ be the
%order of the permutation on $Q$ corresponding to $ab^{-1}$.
%%such that
%%\begin{enumerate}
% %\item $a$ and\/ $b$ act as distinct but commuting
%  %   permutations $\lambda_a$ an $\lambda_b$ of\/~$Q$, with orders $o(a)$ and $o(b)$,
%  %   respectively,
%%\item there exist $s \in Q$ and $d$ dividing $\lcm(o(a),o(b))$ with $sa^db^{\lambda-d}\ne s$
%% \item $\abs \Sigma$ is strictly less than each prime
% %    divisor of\/ $\lcm(o(a),o(b))$.
%%\item
%Assume that for each connected component $C$ of $\M$, the induced actions of
%$\Sigma$ on $C$  does not contain a non-singleton coset of  an $S_C$-subgroup
%whose order divides $h$.
%%that of the action of $\lambda_a \lambda_b^{-1}$
%%and consisting entirely of projection of elements whose order divides $\lcm(o(a),o(b))$.
%%\end{enumerate}
%Then $\M$ is non-dualizable.
\end{theorem}
\begin{proof}
Define $\lambda$ to be the least common multiple of the orders of the
permutations $\rho_b$ and $\rho_c$ of $Q$. Throughout this proof, we blur the
distinction between the elements $b,c$ of $\Sigma$ and the permutations
$\rho_b, \rho_c$ of~$Q$: for $q \in Q$, we write $q \cdot b^{-1}$ to mean $q
\cdot b^{\lambda-1}$ and write $q \cdot c^{-1}$ to mean $q \cdot
c^{\lambda-1}$.

As the permutation $\rho_b(\rho_c)^{-1}$ of $Q$ has order $m > 1$, there are
distinct states $r,s \in Q$ such that
\[
r = s \cdot b c^{-1}.
\]

%By assumption, $s \ne r$.

We shall use the Non-dualizability Lemma~\ref{lem:ND} with $\nu := \abs \Sigma
- 1$. (Note that $\abs \Sigma \ge 2$ and so $\nu \ge 1$.) Define the index set $S := (Q \times \{b,c\}) \cup \N$. For each $i \in
\N$, define $v_i \in Q^S$ by
\[
v_i(q,b) = q, \qquad v_i(q,c) = q \qquad\text{and}\qquad
v_i(j) =
 \begin{cases}
  r &\text{if $j = i$,}\\
  s &\text{otherwise,}
 \end{cases}
\]
for all $q \in Q$ and $j \in \N$. Now define $A_0 := \{\, v_i \mid i \in \N
\,\}$. For each $I \subseteq \N$ with $\abs I = \nu$, define $w_I \in \Sigma^S$
by
\[
w_I(q,b) = b, \qquad w_I(q,c) = c \qquad\text{and}\qquad
w_I(j) =
 \begin{cases}
  b &\text{if $j \in I$,}\\
  c &\text{otherwise,}
 \end{cases}
\]
for all $q \in Q$ and $j \in \N$. Now define $B := \{\, w_I \mid I \subseteq \N
\text{ with } \abs I = \nu\,\}$ and define $A := \sg {\M^S}{A_0 \cup B}$.

\medskip
\emph{Step 1.} We first check condition~\ref{lem:ND}(2). Note that $g \in M^S$
is given by
\[
g(q,b) = q, \qquad g(q,c) = q \qquad\text{and}\qquad
g(j) = s,
\]
for all $q \in Q$ and $j \in \N$. Suppose that $g \in A$. Then we can write
\begin{equation}
g = v_i \cdot w_{I_1} \dotsm w_{I_\ell}\label{eq:g}
\end{equation}
in $\M^S$. By considering equation~\eqref{eq:g} at each coordinate in $Q \times
\{b,c\}$, we infer that $(\rho_b)^\ell =  \id Q = (\rho_c)^\ell$. So $\ell$ is
a multiple of $\lambda$.

\smallskip
\emph{Case 1: $i \notin I_1 \cup \dotsb \cup I_\ell$.} Since $\lambda$
divides~$\ell$, we have $r \cdot c^\ell = r$. But evaluating
equation~\eqref{eq:g} at coordinate~$i$ gives $s = r \cdot c^\ell$, which is a
contradiction.

\smallskip
\emph{Case 2: $i \in I_1 \cup \dotsb \cup I_\ell$.} Enumerate $I_1 \cup \dotsb
\cup I_\ell$ as $i = i_0, i_1, \dotsc, i_k$. For each $j \in \{0,1,\dotsc,k\}$,
let $n_j$ denote the number of occurrences of $i_j$ in the sets
$I_1,\dotsc,I_\ell$. Then $n_0 + n_1 + \dotsb + n_k = \nu \ell$, as the sets
$I_1, \dotsc, I_\ell$ all have size~$\nu$.

Evaluating equation~\eqref{eq:g} at coordinate~$i = i_0$, we get $s = r \cdot
b^{n_0}c^{\ell - n_0}$, since the permutations $\rho_b$ and $\rho_c$ commute.
This gives $s = r \cdot (bc^{-1})^{n_0}$, as $\lambda$ divides~$\ell$. For each
$j \in \{1,\dotsc,k\}$, by evaluating equation~\eqref{eq:g} at coordinate~$i_j$
we get $s = s \cdot b^{n_j}c^{\ell - n_j}$ and so $s = s \cdot
(bc^{-1})^{n_j}$. Since $\lambda$ divides~$\ell$, we now obtain
\begin{align*}
r &= r \cdot (bc^{-1})^{\nu\ell} = r \cdot (bc^{-1})^{n_0 + n_1 + \dotsb + n_k}\\
 &= r \cdot (bc^{-1})^{n_0}(bc^{-1})^{n_1} \dotsm (bc^{-1})^{n_k}
 = s \cdot (bc^{-1})^{n_1} \dotsm (bc^{-1})^{n_k} = s,
\end{align*}
which is a contradiction.

\medskip
\emph{Step 2.} To check condition~\ref{lem:ND}(1), let $x \colon \A \to \M$ be
a homomorphism. By considering three separate cases, we will show that $\ker (x
\rest {A_0})$ has a unique block of size more than~$\nu$. In each case, the
following equation plays a central role:
\begin{equation}
v_i = v_j \cdot w_{K \cup \{i\}}(w_{K \cup \{j\}})^{-1},\label{eq:main}
\end{equation}
for all $i,j \in \N$ and all $K \subseteq \N\comp \{i,j\}$ with $\abs K = \nu -
1$.

\smallskip
\emph{Case 1: $0 \in x(A_0)$.} By equation~\eqref{eq:main}, we must have
$x(A_0) = \{0\}$.

\smallskip
\emph{Case 2: $x(A_0) \cap \Sigma \ne \emptyset$.} By equation~\eqref{eq:main},
we get $x(A_0) = \{0\}$. So this case cannot happen.

\smallskip
\emph{Case 3: $x(A_0) \subseteq Q$.} By equation~\eqref{eq:main}, we have $x(B)
\subseteq \Sigma$ and $x(A_0) \subseteq C$, for some component $C$ of~$\M$. Let
$\gamma \colon \Sigma\to S_C$ map each letter to its action on~$C$; so that
$\gamma(a) = \rho_a \rest C$. Then $\langle \gamma(\Sigma) \rangle$ is a
transitive abelian group of permutations of~$C$, and is therefore regular
(i.e., the stabilizer of each element of $C$ is trivial). For all $w_I \in B
\subseteq \{b,c\}^S$, we have $v_1 = v_1 \cdot (w_I)^\lambda$ in $\A$ and so
$x(v_1) = x(v_1) \cdot x(w_I)^\lambda$ in $\M$; it follows that the order of
the permutation $\gamma(x(w_I))$ of $C$ divides~$\lambda$; this means that, for
$q \in C$, it makes sense to write $q \cdot x(w_I)^{-1}$ to mean $q \cdot
x(w_I)^{\lambda - 1}$.

Assume that $I = \{i_1,\dotsc,i_{\nu+1}\}$ and $J = \{j_1,\dotsc, j_{\nu+1}\}$
are disjoint subsets of~$\N$, each of size~$\nu + 1$, such that
\[
x(v_{i_1}) = x(v_{i_2}) = \dotsb = x(v_{i_{\nu+1}}) \quad\text{and}\quad
x(v_{j_1}) = x(v_{j_2}) = \dotsb = x(v_{j_{\nu+1}}).
\]
It suffices to show that $x(v_{i_1}) = x(v_{j_1})$.

First define a sequence of $\nu + 2$ subsets of~$\N$:
\begin{align*}
I_0 &:= \{i_1, i_2, i_3, \dotsc, i_{\nu+1} \} = I,\\
I_1 &:= \{j_1, i_2, i_3, \dotsc, i_{\nu+1} \},\\
I_2 &:= \{j_1, j_2, i_3, \dotsc, i_{\nu+1} \},\\
&\qquad\qquad \vdots\\
I_{\nu+1} &:= \{j_1, j_2, j_3, \dotsc, j_{\nu+1} \} = J.
\end{align*}
We shall use the following consequence of equation~\eqref{eq:main}:
\begin{equation}
v_{i_1}  \equiv_x  v_{i_{n+1}} =  v_{j_{n+1}} \cdot w_{I_n}(w_{I_{n+1}})^{-1} \equiv_x v_{j_1} \cdot w_{I_n}(w_{I_{n+1}})^{-1},\label{eq:one}
\end{equation}
for all $n  \in \{0,1, \dotsc, \nu\}$. This implies that
\begin{equation*}
x(v_{j_1}) \cdot x(w_{I_0})x(w_{I_1})^{-1} = x(v_{j_1}) \cdot x(w_{I_n})x(w_{I_{n+1}})^{-1},
\end{equation*}
for all $n  \in \{1, \dotsc, \nu\}$. (Recall that the order of each permutation
of $C$ in $\gamma(x(B))$ divides~$\lambda$.) As $\langle \gamma(\Sigma)
\rangle$
%induces a transitive abelian group action on~$C$,
is a regular group of permutations of $C$,
 it follows that
\begin{equation}
h := \gamma\big(x (w_{I_1})\big)^{-1} \circ\gamma\big(x (w_{I_0})\big)
 = \gamma\big(x (w_{I_{n+1}})\big)^{-1} \circ \gamma\big(x (w_{I_n})\big),
\label{eq:multiplier}
\end{equation}
for all $n  \in \{1, \dotsc, \nu\}$.

As $x(B) \subseteq \Sigma$ and $\abs \Sigma = \nu + 1$, we have
$\gamma\big(x(w_{I_k})\big) =\gamma\big(x(w_{I_\ell})\big)$, for some $k <
\ell$. We may choose $k,\ell$ with $\ell-k$ minimal. Now consider the distinct
permutations
\begin{equation}\label{eq:coset}
\gamma\big(x(w_{I_k})\big), \gamma\big(x(w_{I_{k+1}})\big), \dots, \gamma\big(x(w_{I_{\ell-1}})\big)
\end{equation}
of $C$.  By (\ref{eq:multiplier}), we have
$$
\gamma\big(x (w_{I_n})\big) \circ h^{-1} = \gamma\big(x (w_{I_{n+1}})\big),
$$
for all $n  \in \{k, \dotsc, \ell - 1\}$, and
$$
\gamma\big(x (w_{I_{\ell-1}})\big) \circ h^{-1} = \gamma\big(x (w_{I_{\ell}})\big) = \gamma\big(x (w_{I_{k}})\big).
$$
It follows that the permutations in (\ref{eq:coset}) form a coset of the cyclic
subgroup $\langle h \rangle$ of~$S_C$. So the permutation $h$ of $C$ has order
$\ell - k$.

As the permutation $\rho_b(\rho_c)^{-1}$ of $Q$ has order~$m$, we have $v_{1}
\cdot (w_{I_0}(w_{I_1})^{-1})^m = v_{1}$ in $\A$, and therefore
\[
h^m\big(x(v_{1})\big) = x(v_{1}) \cdot \big(x(w_{I_0})x(w_{I_1})^{-1}\big)^m = x(v_{1})
\]
in $\M$.
%Since $h$ has uniformly sized cycles,
As $\langle \gamma(\Sigma) \rangle$ is regular,
this implies that the order of
$h$ divides~$m$. So $\ell - k$ divides~$m$. By assumption, this is only
possible if the subgroup $\langle h \rangle$ of $S_C$ is trivial, which implies
that $h =\text{id}_C$.
Hence $\gamma(x (w_{I_0})) = \gamma(x (w_{I_1}))$, by~(\ref{eq:multiplier}).
Now, using~(\ref{eq:one}), we have
$$
x(v_{j_1}) = x(v_{j_1}) \cdot x(w_{I_0})x(w_{I_1})^{-1}
=  x\big(v_{j_1} \cdot w_{I_0} (w_{I_{1}})^{-1}) = x(v_{i_1}).
$$
So \ref{lem:ND}(1) holds, as required.
\end{proof}

\begin{chatexample}\label{ex:nondcomm}
The simplest example coming from the previous theorem is the
$3$-state automatic algebra $\mathbf C_3 = \langle \{1, 2, 3\} \cup \{b, c\} \cup \{0\}; \cdot\rangle$ shown in
Figure~\ref{fig:nd}. This algebra is non-dualizable by Theorem~\ref{thm:nondcomm}. But $\mathbf C_3$
is not inherently non-dualizable, by Corollary~\ref{thm:dualabel}: a dualizable automatic algebra can be obtained from $\mathbf C_3$ by adding a letter
that acts as the identity.
\end{chatexample}

%%%%%%%%%%%%%%%%%%%%%%%%%%%%%%%%%%%%
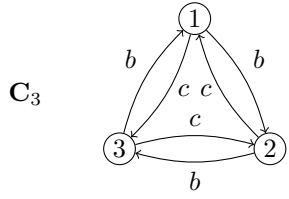
\begin{figure}[t]
\begin{center}
\begin{tikzpicture}
 \node at (0,-1) {$\mathbf C_3$};
 \begin{scope}[xshift=2.25cm]
      \node[state] (1) at (0,0) {$1$};
      \node[state] (2) at ($(1)+(-60:2)$) {$2$};
      \node[state] (3) at ($(1)+(-120:2)$) {$3$};
      \path (1) edge [->, bend left=15] node {$b$} (2);
      \path (2) edge [->, bend left=15] node {$b$} (3);
      \path (3) edge [->, bend left=15] node {$b$} (1);
      \path (1) edge [->, bend left=15, right] node {$c$} (3);
      \path (3) edge [->, bend left=15] node {$c$} (2);
      \path (2) edge [->, bend left=15, left] node {$c$} (1);
   \end{scope}
\end{tikzpicture}
\end{center}
\caption{A non-dualizable automatic algebra}\label{fig:nd}
\end{figure}
%%%%%%%%%%%%%%%%%%%%%%%%%%%%%%%%%%%%

We can now give the promised alternating chain.

\begin{example}\label{ex:chain}
There is an infinite ascending chain $\M_1 \le \M_2 \le \M_3
\le \dotsb$ of finite automatic algebras that are alternately
dualizable and non-dualizable.
\end{example}
\begin{proof}
For each odd prime~$p$, let $\mathbf C_p = \langle \{1,2,\dots,p\} \cup \{b,c\}
\cup \{0\}; \cdot\rangle$ be the $p$-state version of the $3$-state automatic
algebra from Figure~\ref{fig:nd}. We start with $\M_1 := \mathbf C_3$. So
$\M_1$ is non-dualizable, by the previous example.

Now assume $\M_n = \langle Q_n \cup \Sigma_n \cup \{0\}; \cdot
\rangle$ has been defined, for some odd number~$n$, so that
$\Sigma_n$ consists of commuting permutations of~$Q_n$. To
create~$\M_{n+1}$, take $Q_{n+1} := Q_n$ and construct
$\Sigma_{n+1}$ from $\Sigma_n$ by adding enough new
permutations so that $\Sigma_{n+1}$ forms an abelian group of
permutations of~$Q_{n+1}$. Then $\M_{n+1}$ is dualizable, by
Corollary~\ref{thm:dualabel}.

Finally, assume that $\M_n = \langle Q_n \cup \Sigma_n \cup \{0\}; \cdot
\rangle$ has been defined, for some even number~$n$, so that $\Sigma_n$
consists of commuting permutations of~$Q_n$. Choose a prime $p > \abs
{\Sigma_n} + 3$. Define $Q_{n+1} := Q_n \du \{1,2,\dotsc,p\}$ and $\Sigma_{n+1}
:= \Sigma_n \du \{b,c\}$. In $\M_{n+1}$, each letter in $\Sigma_n$ should act
on $Q_n$ as it does in~$\M_n$ and act on $\{1,2,\dotsc,p\}$ as the identity,
and each letter in $\{b,c\}$ should act on $Q_n$ as the identity and act on
$\{1,2,\dotsc,p\}$ as it does in $\mathbf C_p$.
%%Then $\abs {\Sigma_{n+1}}$ is strictly less than $p = o(a) = o(b)$.
So the action corresponding to $bc^{-1}$ has order~$p$.

We will use the previous theorem to check that $\M_{n+1}$ is non-dualizable.
Let $D$ be a component of~$\M_{n+1}$. If $D = \{1,2,\dotsc,p\}$, then $\abs
{\{\, \rho_a \rest D \mid a \in \Sigma_{n+1}\,\}} = 3 < p$. If $D \subseteq
Q_n$, then
\[
\abs {\{\, \rho_a \rest D \mid a \in \Sigma_{n+1}\,\}} \le \abs
{\{\, \rho_a \rest D \mid a \in \Sigma_{n}\,\}} + 1 < p,
\]
since $\abs {\Sigma_{n}} + 1 < p$. It follows that $\M_{n+1}$ is
non-dualizable.
\end{proof}

%%%%%%%%%%%%%%%%%%%%%%%%%%%%%%%%%%%%%%%%%%%%%%%%%%%%%%%%%%%%%%%%%%%%%%%%%%%%%%%%%%%%%%%%%%%%%%%

\section{Appendix}

This appendix contains proofs of the two purely group-theoretic results used in
Section~\ref{sec:trickydual}.

\begin{prop}
Let $H$ be a finite abelian group with exponent dividing $m$ and let $u \in H$. Then there
is a homomorphism $\chi \colon H \to \Z_m$ such that, for all $h \in H\setminus
\{e\}$, there exists $\varphi\in \End H$ with $\varphi(u)=u$ and\/
$\chi(\varphi(h)) \ne 0$.
\end{prop}
\begin{proof}
We prove the claim first for abelian $p$-groups of the form $\mathbb
Z_{p^{n_1}} \times \mathbb Z_{p^{n_2}}$, then for arbitrary finite abelian
$p$-groups, and finally for arbitrary finite abelian groups.

Assume that $H=\Zp{n_1} \times \Zp{n_2}$ with $n_1 \le n_2$, and write
$u=(u_1,u_2)$. Factorize $u_i$ as $a_ip^{k_i}$, where $k_i \le n_i$ and $p
\nmid a_i$. Then $u_i$ has order $p^{d_i}$ in $\Zp{n_i}$, where $d_i :=
n_i-k_i$. Let $\mu\colon \Zp{n_1}\to \Zp{n_2}$ be the homomorphism
$\mu(x)=p^{n_2-n_1}x$. We first show that we can assume with no loss of
generality that $d_1\le d_2$. We use the automorphism $\sigma \colon H\to H$
given by $\sigma((x,y)) = (x,\mu(x)+y)$. If $d_1  > d_2$, then
\begin{align*}
\mu(u_1)+u_2 &=
p^{n_2-n_1}u_1+u_2 = p^{n_2-n_1}a_1p^{k_1} + a_2p^{k_2} \\ &= a_1p^{n_2-d_1} +
a_2p^{n_2-d_2} = p^{n_2-d_1}(a_1 + a_2p^{d_1 - d_2}),
\end{align*}
with $p \nmid (a_1+a_2p^{d_1 - d_2})$; so the image $v=(u_1,\mu(u_1)+u_2)$ of
$u$ under the automorphism $\sigma$ is such that its second coordinate has
order $p^{d_1}$ in $\Zp{n_2}$, equal to the order of its first coordinate.

With the factorization of $H$ thus adjusted, we now take $\chi \colon H \to
\Zp{n_2}$ to be the second projection. Let $h=(h_1,h_2) \in H \comp \{\hat
0\}$. If $h_2\ne 0$, then we can choose $\phi$ to be the identity endomorphism
of~$H$. Assume now that $h_2=0$ and $h_1\ne 0$.  Define $d := d_2-d_1 \ge 0$,
let $b$ be the multiplicative inverse of $a_2$ in $\Zp{n_2}$, and let $c :=
(a_1p^d-a_2)b$. Define $\phi \colon H\to H$ by $\phi((x,y)) = (x,\mu(x) - cy)$.
Obviously $\phi\in\End H$ and $\chi(\phi(h)) = \mu(h_1) \ne 0$, as $\mu$ is
injective.  It remains to check that $\phi(u)=u$ or, equivalently, that
$\mu(u_1)-cu_2=u_2$.  In fact,
\begin{align*}
\mu(u_1)-cu_2 &= p^{n_2-n_1}a_1p^{k_1} - (a_1p^d-a_2)ba_2p^{k_2}\\
&= a_1(p^{n_2-n_1+k_1} - p^{d+k_2}) + a_2p^{k_2} = 0 + u_2,
\end{align*}
where the 0 term arises as $n_2-n_1+k_1 = n_2-d_1 = (n_2-d_2) + (d_2-d_1) =
k_2+d$.

Next, we prove the claim for abelian $p$-groups.  Assume $H = \Zp{n_1} \times
\dots \times \Zp{n_k}$, where $n_1 \le \dotsb \le n_k$. Write $u= (u_1, \dots,
u_k)$, let $p^{d_i}$ be the order of $u_i$ in $\Zp{n_i}$, and define $d :=
\max(d_1,\dots,d_k)$. If $d \ne d_k$, then pick $i<k$ with $d_i=d$ and consider
$\Zp{n_i} \times \Zp{n_k}$.  By the argument for the previous case, we can find
an automorphism of $\Zp{n_i} \times \Zp{n_k}$ that sends $(u_i,u_k)$ to
$(u_i,v)$, where $v$ is of order $p^d$ in $\Zp{n_k}$. Thus by revising the
decomposition of~$H$, we can assume with no loss of generality that $d=d_k$.
Now take $\chi \colon H \to \Zp{n_k}$ to be the $k$th projection. Let $h =(h_1,
\dots, h_k) \in H\setminus \{\hat 0\}$. If $h_k\ne 0$, then we can choose
$\phi$ to be the identity endomorphism of~$H$. So assume that $h_k=0$. Choose
$i<k$ with $h_i\ne 0$. By the argument for the previous case, we can find $\psi
\in \End{ \Zp{n_i} \times \Zp{n_k}}$ such that $\psi((u_i,u_k)) = (u_i,u_k)$
and $\pi_2\circ \psi((h_i,0)) \ne 0$. If we define $\phi$ to act as $\psi$ on
$\Zp{n_i} \times \Zp{n_k}$ and as the identity on the other factors, then we
get our desired endomorphism.

Finally, we prove the claim for arbitrary finite abelian groups.  Let
$p_1,\dotsc,p_k$ be distinct primes and let $H = H_1 \times \dots \times H_k$,
where $H_i$ is an abelian $p_i$-group. Write $u=(u_1,\dotsc,u_k)$. By the
previous case, for each $i \in \{1,\dotsc,k\}$ we can find a homomorphism
$\chi_i \colon H_i \to \mathbb Z_{p_i^{n_i}}$, where $p_i^{n_i} \mid m$, such
that for every $h \in H_i\setminus \{\hat 0\}$ there exists $\phi \in \End
{H_i}$ satisfying $\phi(u_i)=u_i$ and $\chi_i(\phi(h)) \ne 0$. We can now take
$\chi := \chi_1 \sqcap \dots \sqcap \chi_k \colon H \to \mathbb Z_{p_1^{n_1}}
\times \dots \times \mathbb Z_{p_k^{n_k}}$ to be the natural product map.
\end{proof}

While the following basic lemma can be proved using elementary methods, it also
follows immediately from the fact that the cyclic group $\mathbb Z_m$ is
strongly self-dualizing; see~\cite[4.4.2]{NDftWA}.

\begin{lem}
Let $m,k \in \N$ and let $H$ be a subgroup of $(\mathbb Z_m)^k$. Then $H$ can
be described as the set of solutions in $\mathbb Z_m$ to a system of
homogeneous linear equations in $k$ variables with integer coefficients.
\end{lem}
%\begin{proof}
%It is well-known and easily proved that the clone of $\mathbb Z_m$ is
%determined by the relation $\set{(a,b,c) \in (\mathbb Z_m)^3}{$a+b=c \pmod m$}$.
%Hence if $H \leq (\mathbb Z_m)^k$ then there exists a homogeneous system
%$\mathscr S$ of equations in $k+n$ variables (for some $n \geq 0$) and integer
%coefficients such that $H$ is the projection onto the first $k$ variables of the
%set of solutions to $\mathscr S$.  Hence it suffices to show that for every
%such system $\mathscr S(\tup{x},y)$ in $k+1$ variables there exists a system
%$\mathscr S'(\tup{x})$ in $k$ variables which is equivalent to $\exists y
%\mathscr S(\tup{x},y)$.  The key to proving this is the fact that, for any
%factorization $m=de$,
%$\mathbb Z_m$ satisfies $\exists y [z=dy]~\Leftrightarrow~ ez=0$.
%
%Indeed, suppose $\mathscr S(\tup{x},y)$ consists of the equations
%$a_{i,1}x_1 + \cdots + a_{i,k}x_k + b_iy=0$ for $1 \leq i \leq n$.
%Let $d=\gcd(b_1,\ldots,b_n,m)$.  If $d=m$ then the equations do not depend on
%$y$ and thus $H$ is defined by the equations $a_{i,1}x_1 + \cdots+a_{i,k}x_k=0$.
%If $d<m$, then write $m=de$ and $b_i=db_i'$ for each $i$.  Some integer linear
%combination of the original equations has the form
%$c_1x_1 + \cdots + c_kx_k-dy=0$.  We can now form $\mathscr S'(\tup{x})$ as
%follows: in the $i$th original equation replace $b_iy$ with $b_i'\sum_jc_jx_j$,
%redistribute terms and put the resulting equation in $\mathscr S'$.
%Finally, put the equation $e\sum_jc_jx_j=0$ in $\mathscr S'$.
%\end{proof}

\begin{prop}
Assume $M$ is a $j \times k$ matrix over $\mathbb Z_m$ whose rows form a
subgroup of $(\mathbb Z_m)^k$, whose columns form a coset of a subgroup of
$(\mathbb Z_m)^j$, and which is such that every row contains at least one~$0$.
Then some column is constantly~$0$.
\end{prop}
\begin{proof}
Let $H$ be the subgroup of $(\mathbb Z_m)^k$ consisting of the rows of $M$. Let
$E$ be the set of all $\tup{c}=(c_1,\ldots,c_k) \in \mathbb Z^k$ such that the
equation ``$\sum_{i=1}^kc_ix_i=0$'' is satisfied by every member of $H$. Define
$R=\set{\sum_{i=1}^kc_i}{$\tup{c} \in E$}$ and note that $m \in R$ (as $H$
satisfies the equation $mx_1 = 0$). Thus we can define $d=\gcd(R)$.

\medskip

\emph{Case 1}: $d>1$. Choose a prime $p \mid d$.  Then $p \mid m$, so $a :=
m/p$ is a nonzero element of $\mathbb Z_m$.  Thus $(a,a,\dots,a)$ is a solution
of every homogeneous linear equation satisfied by $H$, by the choice of $p$,
and hence $(a,a,\dots,a) \in H$, by the previous lemma.  But this contradicts
the assumption that every row of $M$ contains at least one $0$.

\medskip

\emph{Case 2}: $d=1$. Since $E$ is closed under integer linear combinations,
there exists $\tup{c} \in E$ such that $\sum_{i=1}^k c_i=1$.  Let
$C_1,\dots,C_k$ denote the columns of $M$ and define $D=\sum_{i=1}^k c_iC_i$.
Then $D$ is a column of $M$, since we are assuming the columns of $M$ form a
coset of a subgroup of $(\mathbb Z_m)^j$. But the fact that $\tup{c}\in E$
implies that $D$ is constantly $0$.
\end{proof}

\end{document}